\documentclass{amsart}
\usepackage[normalem]{ulem}

\usepackage{amssymb,epsfig,mathrsfs,mathpazo,stackengine,scalerel,url}
\usepackage{enumerate}
\usepackage{color}
\usepackage{epsfig,bm}
\usepackage[multiple]{footmisc}
\usepackage{accents} 
\usepackage{tikz}
\usetikzlibrary{math}
\usepackage{pgfplots}
\usepackage{csvsimple}
\usepackage{pgfplotstable}

\newtheorem{theorem}{Theorem}[section]
\newtheorem{lemma}[theorem]{Lemma}

\newtheorem{proposition}[theorem]{Proposition}

\newtheorem{corollary}[theorem]{Corollary}

\theoremstyle{definition}

\theoremstyle{remark}
\newtheorem{remark}[theorem]{Remark}
\newtheorem{remarks}[theorem]{Remarks}

\numberwithin{equation}{section}

\newcommand{\R}{\mathbb R}

\newcommand{\cL}{\mathcal L}
\newcommand{\Lis}{\cL\mathrm{is}}

\newcommand{\identity}{\mathrm{Id}}

\DeclareMathOperator{\ran}{ran}
\DeclareMathOperator{\supp}{supp}

\DeclareMathOperator{\diam}{diam}

\DeclareMathOperator{\divv}{div}
\DeclareMathOperator{\curl}{curl}

\DeclareMathOperator{\Span}{span}

\newcommand{\be}{\begin{equation}}
\newcommand{\ee}{\end{equation}}

\stackMath
\newcommand\tenq[2][1]{%
 \def\useanchorwidth{T}%
  \ifnum#1>1%
    \stackunder[0pt]{\tenq[\numexpr#1-1\relax]{#2}}{\scriptscriptstyle\sim}%
  \else%
    \stackunder[1pt]{#2}{\scriptscriptstyle\sim}%
  \fi%
}

\newcommand\tens[2][1]{%
\ThisStyle{\tenq[#1]{\SavedStyle #2}}}

{\par\begin{samepage}%
\begin{tabbing}\ttfamily%
 \hspace*{5mm}\=\hspace{3ex}\=\hspace{3ex}\=\hspace{3ex}\=\hspace{3ex}%
\=\hspace{3ex}\=\hspace{3ex}\=\hspace{3ex}\=\hspace{3ex}\kill}%
{\end{tabbing}\end{samepage}}

\makeatletter
\@namedef{subjclassname@2020}{%
  \textup{2020} Mathematics Subject Classification}
\makeatother

\title[A well-posed FOSLS formulation of the instationary Stokes equations]{A well-posed First Order System Least Squares formulation of the instationary Stokes equations}
\date{\today}

\author{Gregor Gantner and Rob Stevenson}

\address{Korteweg-de Vries (KdV) Institute for Mathematics, University of Amsterdam, P.O. Box 94248, 1090 GE Amsterdam, The Netherlands.}
\email{g.gantner@uva.nl, r.p.stevenson@uva.nl}
\thanks{The first author has been supported by the Austrian Science Fund (FWF) under grant J4379-N. The second author has been supported by NSF Grant DMS 172029.}

\subjclass[2020]{%
35F46, 
35K20, 
35A15, 
65M12, 
65M15, 
65M60, 
76D07} 

\keywords{Instationary Stokes equations, slip boundary conditions, simultaneous space-time variational formulation, First Order System Least Squares (FOSLS), finite elements with a commuting diagram}

\begin{document}

\begin{abstract} In this paper, a well-posed simultaneous space-time First Order System Least Squares formulation is constructed of the instationary incompressible Stokes equations with slip boundary conditions.
As a consequence of this well-posedness, the minimization over any conforming triple of finite element spaces for velocities, pressure and stress tensor gives a quasi-best approximation from that triple. The formulation is practical in the sense that all norms in the least squares functional can be efficiently evaluated. Being of least squares type, the formulation comes with an efficient and reliable a posteriori error estimator. In addition, a priori error estimates are derived, and numerical results are presented.
\end{abstract}

\maketitle

\section{Introduction}
This paper is about the numerical solution of the instationary incompressible Stokes equations in a simultaneous space-time variational formulation.
Compared to classical time-stepping schemes, space-time methods for evolutionary PDEs are much better suited for a massively parallel implementation (e.g.~\cite{234.7,306.65}), allow for local refinements simultaneously in space and time (e.g.~\cite{69.05,249.4,75.38,64.585,249.991,306.6}), produce numerical approximations that are quasi-best in the employed trial spaces, and are advantageous for all applications where the whole time-evolution is simultaneously needed as for optimal control or data-assimilation problems.

The usual bilinear form resulting from a space-time variational formulation is not coercive\footnote{`For parabolic problems `almost' 
coercive bilinear forms based on  $H^\frac12$-type spaces in time are considered in \cite{75.20,64.5765,249.3}.}.
Consequently, trial and test spaces of equal dimension that give a stable Petrov--Galerkin discretization are difficult to construct.
For this reason,   R.~Andreev proposed in~\cite{11} to use minimal residual discretizations for parabolic problems.
They have an equivalent interpretation as Galerkin discretizations of an extended self-adjoint, but indefinite mixed system having as secondary variable
 the Riesz lift of the residual of the primal variable.
  To ensure that the discrete solution is a quasi-best approximation from a pair of trial spaces,
  a remaining restriction is the Ladyshenskaja--Babu\u{s}ka--Brezzi (LBB) stability of this pair.
  
In \cite{75.257}, T.~F\"{u}hrer and M.~Karkulik showed well-posedness of a First Order System Least Squares (FOSLS) formulation of the heat equation that was proposed earlier in \cite{23.5}.
In the terminology from \cite{23.5}, this formulation is `practical' in the sense that the different components of the residual are all measured in $L_2$-type norms, which, unlike fractional Sobolev norms or dual norms, are efficiently computable.
The minimizer of this least squares functional over \emph{any} finite dimensional subspace of the solution space is a quasi-best approximation from that subspace.
Further results about this FOSLS can be found in \cite{75.28}.

The topic of the current work is to find a similar FOSLS formulation of the instationary incompressible Stokes equations posed on a time-space cylinder $(0,T) \times \Omega$.
The velocity component of the solution of these Stokes equations solves a parabolic evolution equation.
To use this fact as the basis of a numerical approximation scheme, however, one would have to construct trial spaces of divergence free velocities, which even in two space-dimensions is not very practical, and furthermore it does not yield an approximation for the pressure.

Well-posed space-time variational formulations of the instationary incompressible Stokes equations of second order for the \emph{pair} of velocities and pressure were given in \cite{77.3,247.155}.
So far these formulations have not been used for the construction of a numerical solution method. 
From these references it appears that no-slip and slip boundary conditions require a quite different treatment. 
In the current work, which is partly based on insights from \cite{77.3}, we consider slip boundary conditions.
An important work on the numerical treatment of slip boundary conditions in the setting of the 
(stationary) incompressible (Navier--)Stokes equations is \cite{308.3}.

Although in our FOSLS one component of the residual, viz.\ the (spatial) divergence  $\divv_x \tens{u}_\delta$
of the approximating velocity-field $\tens{u}_\delta$
 is not measured in the $L_2$-norm but in the $H^1((0,T);L_2(\Omega))$-norm, since also the latter norm can be efficiently evaluated the method is practical in the aforementioned sense.
 Since our proof of well-posedness of the FOSLS requires $L_2$-boundedness of the so-called Leray projector, our proof is restricted to domains $\Omega$ that are convex or have a $C^2$-boundary.
Numerical experiments that we will report on indicate that well-posedness is indeed lost when the $H^1((0,T);\tens{L}_2(\Omega))$-norm is replaced by the $L_2((0,T)\times\Omega)$-norm, and furthermore that likely our FOSLS is not (`fully') well-posed on an L-shaped domain.

Although efficiently evaluable, the term $\|\divv_x \tens{u}_\delta\|_{H^1((0,T);\tens{L}_2(\Omega))}$ in the least-squares functional causes two problems.
First, for $\tens{u}_\delta$ being a continuous piecewise polynomial w.r.t.\ an arbitrary polytopal partition of the time-space cylinder $(0,T) \times \Omega$, generally $\divv_x \tens{u}_\delta \not\in H^1((0,T);\tens{L}_2(\Omega))$. Since on the other hand $C^1$-finite elements are not very practical, as approximation spaces for the 
velocity-field $\tens{u}$, the pressure $p$, and the symmetric stress tensor $\tens[2]{w}:=\nu (\tens[2]{\nabla}_x \tens{u}+(\tens[2]{\nabla}_x \tens{u})^\top)-p\identity$, we consider finite element spaces w.r.t.\ (joint) partitions of $(0,T) \times \Omega$ into \emph{prismatic elements}.
Then global continuity of the functions in the trial space for the velocities suffices for this space to be conforming.
Secondly, when employing standard piecewise polynomial trial spaces w.r.t.~aforementioned partitions, the term $\|\divv_x \tens{u}_\delta\|_{H^1((0,T);\tens{L}_2(\Omega))}$
can be expected to restrict the convergence rate.
For that reason we will enrich such trial spaces with bubble functions as introduced in \cite{38.77}, which allow the construction of a quasi-interpolator 
that has a commuting diagram property, and so maps divergence-free functions to divergence-free functions.
\medskip

This paper is organized as follows: In the remainder of the current section we introduce some notations, and specify the instationary Stokes problem with slip boundary conditions.

In Section~\ref{sec:FOSLS} we present our FOSLS formulation of this Stokes problem and show that it is well-posed.
Initially, we consider a space for the pressure of functions that for each time $t$ have a zero average over the spatial domain. Since such a constraint is difficult to incorporate into a trial space of piecewise polynomials w.r.t.~a partition of the time-space cylinder that cannot be subdivided into time slabs, later we extend our FOSLS by adding this constraint to the least-squares functional.

In Section~\ref{sec:a priori} we construct suitable finite-dimensional trial spaces in the solution space of our FOSLS, and give a priori bounds for the error attained by the minimizer of the least squares functional over these trial spaces. Although not essential, here we restrict ourselves to two-dimensional spatial domains $\Omega$ and lowest order finite element spaces.
Furthermore, we consider only quasi-uniform partitions that are products of partitions of the temporal and the spatial domain, so that the finite element spaces are of tensor product form. Insisting on prismatic elements, locally refined partitions necessarily will contain `hanging nodes', which will be addressed in future work.

Aiming at minimizing regularity conditions needed to ensure a certain convergence rate, also for the approximation of the stress tensor $\tens[2]{w}$
we want to employ a trial space in the spatial direction that has an appropriate commuting diagram property.
We use the trial space introduced in \cite{38.78}, which has only 9 DOFs per element, and show such a commuting diagram property.

In Section~\ref{sec:numerics}, we report on several numerical experiments that illustrate that our FOSLS is also quantitatively well-conditioned, that likely the conditions on the spatial domain for well-posedness of our FOSLS can generally not be omitted,
that measuring the term $\divv_x \tens{u}_\delta$ in the $H^1((0,T);\tens L_2(\Omega))$-norm is indeed necessary,
and that likely well-posedness of our FOSLS does not generalize to no-slip boundary conditions.


\subsection{Notation}\label{sec:notation}
In this work, by $C \lesssim D$ we will mean that $C$ can be bounded by a multiple of $D$, independently of parameters on which C and D may depend. 
Obviously, $C \gtrsim D$ is defined as $D \lesssim C$, and $C\eqsim D$ as $C\lesssim D$ and $C \gtrsim D$.

For normed linear spaces $E$ and $F$, we will denote by $\cL(E,F)$ the normed linear space of bounded linear mappings $E \rightarrow F$, 
and by $\Lis(E,F)$ its subset of boundedly invertible linear mappings $E \rightarrow F$.
We write $E \hookrightarrow F$ to denote that $E$ is continuously embedded into $F$.
For simplicity only, we exclusively consider linear spaces over the scalar field $\R$.

For a Hilbert space $W$ that is dense and continuously embedded in a Hilbert space equipped with $L_2(\Sigma)^n$-norm, often we find it convenient to use the scalar product on $L_2(\Sigma)^n$ to denote its unique extension to the duality pairing on $W' \times W$.

We adopt the convention that an undertilde denotes vector-valued functions, operators, and their associated spaces. 
Double undertildes are used for matrix-valued functions, operators, and associated spaces. Here vectors and matrices are of size $n$ and $n \times n$, respectively, with $n$ being the dimension of the `spatial domain' $\Omega$. With $\tens[2]{\mathbb{S}}$ we denote the linear space of symmetric $n \times n$ matrices.
\subsection{Stokes problem with slip conditions}
Let $\Omega \subset \R^n$ be a bounded Lipschitz domain, and $I:=(0,T)$ for some $T>0$.
Given vector fields $\tens{u}_0$ on $\Omega$ and $\tens{f}$ on $I\times \Omega$, a function $g$ on $I \times \Omega$, and a viscosity $\nu>0$, we consider the instationary  Stokes problem with \emph{slip boundary} conditions of finding a velocity field
$\tens{u}$ and corresponding pressure $p$ 
that satisfy
\be \label{free-slip-equations}
\left\{
\begin{array}{@{}r@{}c@{}ll}
\tens{\partial}_t \tens{u} -\nu \tens{\Delta}_{x} \tens{u} +\nabla_{\bf x}\, p & \,\, = \,\,& \tens{f} &\text{ in } I \times \Omega,\\
\divv_x \tens{u}& \,\, = \,\,& g & \text{ in } I\times \Omega,\\
\tens{u} \cdot \tens{n}& \,\, = \,\,& 0 & \text{ on } I \times \partial \Omega,\\
(\identity-\tens{n} \tens{n}^\top)\tens[2]{T}(\nu \tens{u},p) \tens{n}& \,\, = \,\,&0&\text{ on } I\times \partial \Omega,\\
\tens{u}(0,\cdot) & \,\, = \,\,& \tens{u}_0 & \text{ on }  \Omega.
\end{array}
\right.
\ee
Here $(0,\tens{n}) \in \R^{n+1}$ denotes the outward pointing normal vector on $I \times \partial\Omega$,
and for $\tens{v}\colon I \times \Omega \rightarrow \R^n$ and $q\colon I \times \Omega \rightarrow\R$, the \emph{deformation tensor} and \emph{stress tensor} are defined by 
$\tens[2]{D}(\tens{v}):= \tens[2]{\nabla}_x \tens{v}+(\tens[2]{\nabla}_x \tens{v})^\top$ and
$\tens[2]{T}(\tens{v},q):=\tens[2]{D}(\tens{v})-q \identity$, respectively. 
Notice that the second boundary condition means that for all $\tens{\tau} \perp \tens{n}$, $(\tens[2]{D}(\tens{u}) \tens{n}) \cdot  \tens{\tau}=(\tens[2]{T}(\nu \tens{u},p) \tens{n}) \cdot  \tens{\tau}=0$
on $I \times\partial\Omega$.

Using that for a $\tens{v}$ that is twice continuously partially differentiable w.r.t.~the spatial variables, it holds that
$$
\tens{\divv}_x \tens[2]{D}(\tens{v})=\tens{\Delta}_{x} \tens{v} +\tens{\nabla_x} \divv_x \tens{v},
$$
we infer that a classical solution $(\tens{u} ,p)$ of \eqref{free-slip-equations} with $\tens[2]{w}:=-\tens[2]{T}(\nu\tens{u},p)$
satisfies
\be \label{defG}
{\bf G}(\tens{u},\tens[2]{w},p):=(\tens[2]{w}+\tens[2]{T}(\nu\tens{u},p), \tens{\partial}_t \tens{u} +\tens{\divv}_x \tens[2]{w},\divv_x \tens{u},\tens{u}(0,\cdot))=(0,\tens{f}+\nu \tens{\nabla}_x g,g,\tens{u}_0),
\ee
as well as $\tens{u} \cdot \tens{n}= 0$ on $I \times \partial \Omega$, and 
$(\identity-\tens{n} \tens{n}^\top)\tens[2]{w}\tens{n}=0$ on $I\times \partial \Omega$.
We will show that \eqref{defG} gives rise to a well-posed \emph{First Order System Least Squares} (FOSLS) formulation of \eqref{free-slip-equations}.
This means that ${\bf G}$ will be shown to be a boundedly invertible operator from some solution space onto its \emph{range} in the data space $\tens[2]{L}_2(I \times \Omega) \times \tens{L}_2(I \times \Omega)\times H^1(I;L_{2,0}(\Omega)) \times \tens{L}_2(\Omega)$.

\section{First order system least squares} \label{sec:FOSLS}

\subsection{The solution space for our FOSLS} \label{sec:solution-space}
We define
\begin{align*}
L_{2,0}(\Omega)&:=\{p \in L_2(\Omega)\colon \int_\Omega p \,dx=0\},
\intertext{being, by our assumption of $\Omega$ being bounded, a closed subspace of $L_2(\Omega)$, and} 
\tens{\mathbb{H}}^1(\Omega)&:=\{\tens{u} \in \tens{H}^1(\Omega) \colon \tens{u}\cdot\tens{n}=0 \text{ on } \partial\Omega \}.
\end{align*}
Knowing that $\tens{u} \mapsto \tens{u}|_{\partial\Omega}$ is in $\cL( \tens{H}^1(\Omega),\tens{H}^{\frac12}(\partial\Omega))$, and thus in 
$\cL( \tens{H}^1(\Omega),\tens{L}_2(\partial\Omega))$, from $\tens{n} \in L_\infty(\partial\Omega)$ we infer that
$\tens{u} \mapsto \tens{u}|_{\partial\Omega} \cdot\tens{n} \in \cL(\tens{H}^1(\Omega),L_2(\partial\Omega))$, and thus that $\tens{\mathbb{H}}^1(\Omega)$ is a closed subspace of $\tens{H}^1(\Omega)$.
We set
\begin{align*}
\mathscr{Z}:=\{(\tens{u},\tens[2]{w}) \in L_2(I;\tens{\mathbb{H}}^1(\Omega))&\times L_2(I;L_2(\Omega;\tens[2]{\mathbb{S}})) \colon
\tens{\partial}_t \tens{u}+\tens{\divv}_x \tens[2]{w} \in \tens{L}_2(I \times \Omega),\\
&\divv_x \tens{u} \in H^1(I;L_{2,0}(\Omega)), (\identity-\tens{n} \tens{n}^\top)\tens[2]{w}|_{I\times \partial \Omega}\tens{n}=0\},
\end{align*}
equipped with the (squared) graph norm
\begin{align*}
\|(\tens{u},\tens[2]{w})\|_\mathscr{Z}^2:=&\|\tens{u}\|^2_{L_2(I;\tens{\mathbb{H}}^1(\Omega))}+\|\tens[2]{w}\|^2_{\tens[2]{L}_2(I \times \Omega)}\\
&+\|\tens{\partial}_t \tens{u}+\tens{\divv}_x \tens[2]{w} \|_{ \tens{L}_2(I \times \Omega)}^2+\|\divv_x \tens{u}\|_{H^1(I;L_{2,0}(\Omega))}^2.
\end{align*}
The solution space of our FOSLS will be $\mathscr{Z} \times L_2(I;L_{2,0}(\Omega))$. We show that $\mathscr{Z}$, and so the solution space are Hilbert spaces.

\begin{lemma} $\mathscr{Z}$ is a Hilbert space.
\end{lemma}

\begin{proof}
Realizing that $(\tens{\partial}_t \tens{u}+\tens{\divv}_x \tens[2]{w})_i$ is the divergence of the vector field $(\tens{u}_i,\tens[2]{w}_{i\cdot})^\top$, one infers that 
$$
\mathscr{Z}_1:=\{(\tens{u},\tens[2]{w})\in \tens{L}_2(I \times \Omega)\times \tens[2]{L}_2(I \times \Omega)\colon
\tens{\partial}_t \tens{u}+\tens{\divv}_x \tens[2]{w} \in \tens{L}_2(I \times \Omega)\},
$$
equipped with the graph norm, is a Hilbert space. 

Consider $\tens{V}:=H^1_0(I;\tens{L}_2(\Omega)) \cap L_2(I;\tens{\mathbb{H}}^1(\Omega))$, its closed subspace
$\tens{V}_0:=\{\tens{v} \in \tens{V}\colon  \tens{v}|_{I \times \partial\Omega}=0\}$, and the quotient space $\tens{V}/\tens{V}_0$ equipped with $\|[\tens{v}]\|_{\tens{V}/\tens{V}_0}:=\inf\{\|\tens{v}+\tens{v}_0\|_{\tens{V}}\colon \tens{v}_0\in \tens{V}_0\}$.
For smooth $(\tens{u},\tens[2]{w}) \in \mathscr{Z}_1$ and $\tens{v} \in \tens{V}$, integration-by-parts shows that
\begin{align}
\notag
\int_I \int_{\partial\Omega} 
(\identity-\tens{n} \tens{n}^\top)\tens[2]{w}\tens{n} \cdot \tens{v} \,ds \,dt&=
\int_I \int_{\partial\Omega} 
\tens[2]{w}\tens{n} \cdot \tens{v} \,ds \,dt+\int_{\Omega} (\tens{u}\cdot \tens{v})(T,\cdot)- (\tens{u}\cdot \tens{v})(0,\cdot)\,dx\\
\label{eq:partial integration}
&=\int_I  \int_\Omega (\tens{\partial}_t \tens{u}+\tens{\divv}_x \tens[2]{w})\cdot \tens{v} +\tens{u}\cdot \tens{\partial}_t\tens{v}+\tens[2]{w} : \tens[2]{\nabla}_x \tens{v} \,dx\,dt,
\end{align}
 and so $|\int_I \int_{\partial\Omega} 
(\identity-\tens{n} \tens{n}^\top)\tens[2]{w}\tens{n} \cdot \tens{v} \,ds \,dt|
\lesssim \|(\tens{u},\tens[2]{w})\|_{\mathscr{Z}_1} \|[\tens{v}]\|_{\tens{V}/\tens{V}_0}$. We conclude that
$(\tens{u},\tens[2]{w}) \mapsto (\identity-\tens{n} \tens{n}^\top)\tens[2]{w}\tens{n}\in \cL(\mathscr{Z}_1, (\tens{V}/\tens{V}_0)')$, so that
\be \label{H1}
\mathscr{Z}_2:=\{(\tens{u},\tens[2]{w})\in  \mathscr{Z}_1\colon (\identity-\tens{n} \tens{n}^\top)\tens[2]{w}|_{I\times \partial \Omega}\tens{n}=0\}
\ee
is a Hilbert space.

Because furthermore $(\tens{u},\tens[2]{w})\mapsto \tens[2]{w}-\tens[2]{w}^\top \in \cL(\mathscr{Z}_2, \tens[2]{L}_2(I \times \Omega))$, we have that
$$
\mathscr{Z}_3:=\{(\tens{u},\tens[2]{w})\in  \mathscr{Z}_2\colon \tens[2]{w} \in L_2(I;L_2(\Omega;\tens[2]{\mathbb{S}}))\}
$$
is a Hilbert space.

Knowing that $(\tens{u},\tens[2]{w})\mapsto\tens{\partial}_t \tens{u}+\tens{\divv}_x \tens[2]{w}$ is a closed map from its domain in $\tens{L}_2(I \times \Omega)\times \tens[2]{L}_2(I \times \Omega)$ into $ \tens{L}_2(I \times \Omega)$, it is a closed map from its domain in $L_2(I; \tens{\mathbb{H}}^1(\Omega))\times \tens[2]{L}_2(I \times \Omega)$ into $ \tens{L}_2(I \times \Omega)$, and so
$$
\mathscr{Z}_{4}:=\mathscr{Z}_{3}\cap   L_2(I; \tens{\mathbb{H}}^1(\Omega))\times \tens[2]{L}_2(I \times \Omega) 
$$
is a Hilbert space.

Let $(\tens{u}_k,\tens[2]{w}_k)$ a converging sequence in $\mathscr{Z}_{4}$ with limit $(\tens{u},\tens[2]{w})$ for which $\divv_x \tens{u}_k \rightarrow z$ in $H^1(I;L_{2,0}(\Omega))$. From $\tens{u}_k \rightarrow \tens{u}$ in $L_2(I; \tens{\mathbb{H}}^1(\Omega))$, we have $\divv_x \tens{u}_k \rightarrow \divv_x \tens{u}$ in $L_2(I \times \Omega)$, and 
apparently $z= \divv_x \tens{u}$. We conclude that 
$\mathscr{Z}$
is a Hilbert space.
\end{proof}

\begin{remark} \label{rem:density}
With $\mathcal{Z}_2:=\{(\tens{u},\tens[2]{w}) \in \tens{\mathcal D}(\overline{I \!\times\! \Omega})\!\times\! \tens[2]{\mathcal D}(\overline{I \! \times\! \Omega})\colon (\identity-\tens{n} \tens{n}^\top)\tens[2]{w}|_{I\times \partial \Omega}\tens{n}\!=\!0\}$ 
below we will need that $\overline{\mathcal{Z}_2}=\mathscr{Z}_2$. To show this we follow the arguments used in \cite[Thm.~2.4]{75.4}.
Let $\ell \in \mathscr{Z}_2'$ be arbitrary. There exists a $(\tens{\hat{u}},\tens[2]{\hat{w}}) \in \mathscr{Z}_2$ with
 $$
 \ell(\tens{u},\tens[2]{w})=\int_I \int_\Omega \tens{u} \cdot \tens{\hat{u}}+\tens[2]{w} \cdot \tens[2]{\hat{w}}+
(\tens{\partial}_t \tens{u}+\tens{\divv}_x \tens[2]{w})\hat{k}\,dx\,dt,
$$
where $\tens{\hat{k}}:=\tens{\partial}_t \tens{\hat{u}}+\tens{\divv}_x \tens[2]{\hat{w}}$. Suppose that $\ell$ vanishes on $\mathcal{Z}_2$. Then in particular it vanishes on $\tens{\mathcal D}(I \times \Omega)\times \tens[2]{\mathcal D}(I \times\Omega)$, from which it follows that $\tens{\hat{u}}=\partial_t \tens{\hat{k}}$ and $\tens[2]{\hat{w}}=\tens[2]{\nabla}_x \tens{\hat{k}}$, and so $\tens{\hat{k}} \in \tens{H}^1(I \times \Omega)$. Now using that  $\ell$ vanishes on $\mathcal{Z}_2$, integration-by-parts shows that $\tens{\hat{k}}(0,\cdot)=0=\tens{\hat{k}}(T,\cdot)$ and $\tens{\hat{k}}\cdot \tens{n}=0$ on $I \times\partial\Omega$, i.e., $\tens{\hat{k}} \in \tens{V}$. Now from \eqref{eq:partial integration} it follows that $\ell$ vanishes on the whole of $\mathscr{Z}_2$. We conclude that $\overline{\mathcal{Z}_2}=\mathscr{Z}_2$.
\end{remark}

The following lemma shows that for $(\tens{u},\tens[2]{w}) \in \mathscr{Z}$ it holds that $\partial_t \tens{u} \in L_2(I;\tens{\mathbb{H}}^1(\Omega)')$, which will be relevant later.


\begin{lemma} \label{lem2} It holds that $(\tens{u},\tens[2]{w})  \mapsto \tens{u} \in \cL\big(\mathscr{Z},H^1(I;\tens{\mathbb{H}}^1(\Omega)')\big)$.
\end{lemma}

\begin{proof} For $(\tens{u},\tens[2]{w}) \in \mathcal{Z}_2$ and 
$\tens{v} \in L_2(I;\tens{\mathbb{H}}^1(\Omega))$,
one has
\be \label{4}
\begin{split}
\int_I \int_\Omega \tens[2]{w}:\tens[2]{\nabla}_x \tens{v}\,dx\,dt&=-\int_I \int_\Omega \tens{\divv}_x \tens[2]{w} \cdot  \tens{v}\,dx\,dt+
\int_I \int_{\partial\Omega} \tens[2]{w}\tens{n} \cdot \tens{v} \,ds\,dt\\
&=-\int_I \int_\Omega \tens{\divv}_x \tens[2]{w} \cdot  \tens{v}\,dx\,dt.
\end{split}
\ee
i.e., $ \tens{\divv}_x \tens[2]{w}$ extends to an element in $L_2(I;\tens{\mathbb{H}}^1(\Omega)')$, with norm bounded by $\|\tens[2]{w}\|_{\tens[2]{L}_2(I \times \Omega)}$. 
We conclude that
\begin{align*}
\|\tens{\partial}_t \tens{u}\|_{L_2(I;\tens{\mathbb{H}}^1(\Omega)')} &\leq 
\|\tens{\partial}_t \tens{u}+\tens{\divv}_x \tens[2]{w}\|_{L_2(I;\tens{\mathbb{H}}^1(\Omega)')}+
\|\tens{\divv}_x \tens[2]{w}\|_{L_2(I;\tens{\mathbb{H}}^1(\Omega)')}\\
& \leq 
\|\tens{\partial}_t \tens{u}+\tens{\divv}_x \tens[2]{w}\|_{\tens{L}_2(I\times \Omega)}+
\|\tens[2]{w}\|_{\tens[2]{L}_2(I \times \Omega)}.
\end{align*}
By $\overline{\mathcal{Z}_2}=\mathscr{Z}_2$ this result extends to all $(\tens{u},\tens[2]{w}) \in \mathscr{Z}_2$, and thus to all  $(\tens{u},\tens[2]{w}) \in \mathscr{Z}$.
\end{proof}

\subsection{Parabolic PDE on the space of divergence-free velocities} \label{sec:div-free}
In order to be able to show that
\be \label{17}
 \|(\tens{u},\tens[2]{w},p)\|_{\mathscr{Z} \times L_2(I;L_{2,0}(\Omega))}
 \lesssim \|{\bf G}(\tens{u},\tens[2]{w},p)\|_{\tens[2]{L}_2(I \times \Omega) \times \tens{L}_2(I \times \Omega)\times H^1(I;L_{2,0}(\Omega)) \times \tens{L}_2(\Omega)},
 \ee
 we will use that a divergence-free velocity field $\tens{u}$ is the solution of a well-posed parabolic evolution problem. 

We define
\begin{align*}
\tens{\mathcal H}^1(\Omega)&:=\{\tens{u} \in \tens{\mathbb{H}}^1(\Omega) \colon \divv \tens{u} =0 \},\\
\tens{\mathcal H}^0(\Omega)&:=\{\tens{u} \in \tens{L}_2(\Omega) \colon \divv \tens{u} =0,\, \tens{u}\cdot\tens{n}=0 \text{ on } \partial\Omega\}.
\end{align*}
Then $\tens{\mathcal H}^1(\Omega) \hookrightarrow \tens{\mathcal H}^0(\Omega)$ with dense embedding, see, e.g., \cite[Thm.~2.8]{75.5}, and so
$\tens{\mathcal H}^1(\Omega) \hookrightarrow \tens{\mathcal H}^0(\Omega) \simeq \tens{\mathcal H}^0(\Omega)'  \hookrightarrow\tens{\mathcal H}^1(\Omega)'$ forms a Gelfand triple.

For almost all $t \in I$, let $a(t;\cdot,\cdot)$ be a bilinear form on $\tens{\mathcal H}^1(\Omega) \times \tens{\mathcal H}^1(\Omega)$ such that 
for any $\tens \mu,\tens \lambda\in \tens{\mathcal H}^1(\Omega)$, $t\mapsto a(t;\tens \mu,\tens \lambda)$ is measurable on $I$, and such that 
for some constant $\varrho \geq 0$,
for a.e.~$t \in I$, and all $\tens \mu,\tens \lambda$,
\begin{alignat}{2} \label{boundedness}
|a(t;\tens \mu,\tens \lambda)| 
& \lesssim  \|\tens \mu\|_{\tens{\mathcal H}^1(\Omega)} \|\tens\lambda\|_{\tens{\mathcal H}^1(\Omega)}  \quad &&\text{({\em boundedness})},
\\ \label{garding}
 a(t;\tens \mu,\tens \mu)  +\varrho \|\tens \mu\|_{\tens{\mathcal H}^0(\Omega)} ^2 &\gtrsim \|\tens \mu\|^2_{\tens{\mathcal H}^1(\Omega)}  \quad &&\text{({\em G{\aa}rding inequality})}.
\end{alignat}

A proof of the following result can be found, e.g.\ in \cite[Ch.~IV, \S26]{314.91}, \cite[Ch.~XVIII, \S3]{63}, or \cite{247.15}.
\begin{theorem}\label{thm1} Assuming \eqref{boundedness} and \eqref{garding}, the operator $B$ defined by
$$
(B\tens{u})(\tens{v}):= \int_I \int_\Omega \partial_t \tens{u}(t,x) \cdot \tens{v}(t,x) \,dx
+a(t;u(t),v(t)) dt,
$$
satisfies 
$$
\left[\begin{array}{@{}c@{}} B \\ \gamma_0\end{array} \right]\in \Lis\big(L_2(I;\tens{\mathcal H}^1(\Omega)) \cap H^1(I;\tens{\mathcal H}^1(\Omega)'),L_2(I;\tens{\mathcal H}^1(\Omega))' \times \tens{\mathcal H}^0(\Omega)\big),
$$
where $\gamma_0:= \tens{u}\mapsto \tens{u}|_{t=0}$.
\end{theorem}

To motivate the following Proposition~\ref{prop1}, we summarize the main steps of the forthcoming proof of \eqref{17}  in Section~\ref{Swellposedness}.
Consider a triple $(\tens{u},\tens[2]{w},p) \in \mathscr{Z} \times L_2(I;L_{2,0}(\Omega))$, for simplicity here with $\divv_x \tens{u}=0$.
Setting 
$$
\tens{f}(\tens{v}):=\int_I \int_\Omega \tens{\partial}_t \tens{u} \cdot \tens{v}+\tfrac12 \tens[2]{T}(\nu \tens{u},p):\tens[2]{D}(\tens{v})\,dx\,dt,
$$
integration-by-parts \eqref{4} and the triangle inequality show that
\be \label{18}
\|\tens{f}\|_{L_2(I;\tens{\mathbb{H}}^1(\Omega)')}\lesssim
\|\tens[2]{w}+\tens[2]{T}(\nu \tens{u},p)\|_{\tens[2]{L}_2(I\times \Omega)}+
\|\tens{\partial}_t \tens{u}+\tens{\divv}_x \tens[2]{w}\|_{\tens{L}_2(I\times \Omega)}.
\ee
Since $\tens{f}(\tens{v})=\int_I \int_\Omega \tens{\partial}_t \tens{u} \cdot \tens{v}+\tfrac{\nu}{2} \tens[2]{D}(\tens{u}):  \tens[2]{D}(\tens{v}) \,dx\,dt$ for $\tens v \in L_2(I;\tens{\mathcal H}^1(\Omega))$, an application of Theorem~\ref{thm1} will give
\be \label{19}
\|\tens{u}\|_{L_2(I;\tens{\mathcal H}^1(\Omega))}+\|\tens{\partial}_t \tens{u}\|_{L_2(I;\tens{\mathcal H}^1(\Omega)')}
\lesssim \|\tens f\|_{L_2(I;\tens{\mathcal H}^1(\Omega)')}+\|\tens u(0,\cdot)\|_{\tens{\mathcal H}^0(\Omega)},
\ee
whereas for $\tens{v} \in L_2(I;\tens{\mathbb{H}}^1(\Omega))$,
$$
\int_I \int_\Omega p \divv_x\tens{v} \,dx\,dt=\int_I \int_\Omega \tens{\partial}_t \tens{u}\cdot \tens{v} +\tfrac{\nu}{2} \tens[2]{D}(\tens{u}): \tens[2]{D}(\tens{v})\,dx\,dt-\tens{f}(\tens{v}).
$$
Using that 
$
\inf_{0 \neq q \in L_{2,0}(\Omega)} \sup_{0 \neq \tens{v} \in \tens{\mathbb{H}}^1(\Omega)}
\frac{|\int_\Omega q \divv_x\tens{v} \,dx|}{\|q\|_{ L_{2,0}(\Omega)} \|\tens{v}\|_{\tens{\mathbb{H}}^1(\Omega)}} >0,
$
one arrives at
\be \label{20}
\|p\|_{L_2(I;L_{2,0}(\Omega))}  \lesssim \| \tens{\partial}_t \tens{u}\|_{L_2(I;\tens{\mathbb{H}}^1(\Omega)')}+\|\tens{u}\|_{L_2(I;\tens{\mathbb{H}}^1(\Omega))}+
\|\tens f\|_{L_2(I;\tens{\mathbb{H}}^1(\Omega)')}.
\ee
Since $\|\tens[2]{w}\|_{\tens[2]{L}_2(I \times \Omega)} \lesssim \|\tens[2]{w}+\tens[2]{T}(\nu \tens{u},p)\|_{\tens[2]{L}_2(I \times \Omega)}+\|\tens{u}\|_{L_2(I;\tens{\mathbb{H}}^1(\Omega))}+\|p\|_{L_2(I;L_{2,0}(\Omega))}$,
the combination of \eqref{20}, \eqref{19}, and, 
 \eqref{18} shows \eqref{17} \emph{if}, in addition,  $\|\tens{\partial}_t \tens{u}\|_{L_2(I;\tens{\mathbb{H}}^1(\Omega)')} \lesssim \|\tens{\partial}_t \tens{u}\|_{L_2(I;\tens{\mathcal H}^1(\Omega)')}$, which is shown in Proposition~\ref{prop1}.

\begin{remark}The above relates to a discussion in \cite[Sect.~1.1]{5.7} where, in the stationary case (and, actually, for no-slip boundary conditions), it is explained that forcings $\tens f$ should be interpreted as functionals in $\tens{\mathcal H}^1(\Omega)'$ when it concerns their influence on the velocity, whereas the pressure depends on their interpretation as elements of $\tens{\mathbb{H}}^1(\Omega)'$.
\end{remark}

\begin{proposition} \label{prop1} Assume that the $\tens{L}_2(\Omega)$-orthogonal projector $\tens{\Pi}$ onto $\tens{\mathcal H}^0(\Omega)$, also known as the \emph{Leray-projector}, satisfies
$\tens{\Pi} \in \cL(\tens{\mathbb{H}}^1(\Omega),\tens{\mathbb{H}}^1(\Omega))$.
Then for $\tens{u} \in L_2(I;\tens{\mathcal H}^1(\Omega)) \cap H^1(I;\tens{\mathcal H}^1(\Omega)')$, it holds that
$$
\|\tens{\partial}_t \tens{u}\|_{L_2(I;\tens{\mathbb{H}}^1(\Omega)')} \eqsim \|\tens{\partial}_t \tens{u}\|_{L_2(I;\tens{\mathcal H}^1(\Omega)')}.
$$
\end{proposition}

\begin{proof} 
The inequality `$\gtrsim$'. trivially follows from the inclusion $\tens{\mathcal H}^{1}(\Omega)\subset \tens{\mathbb{H}}^1(\Omega)$.
For the other inequality, let $\tens{u} \in H^1(I;\tens{\mathcal H}^{1}(\Omega))$.
It holds that
\begin{align*}
\|\tens{\partial}_t\tens{u}\|_{L_2(I;\tens{\mathbb{H}}^1(\Omega)')}
&=\sup_{0 \neq \tens{v} \in L_2(I;\tens{\mathbb{H}}^1(\Omega))} \frac{\int_I \int_\Omega \tens{\partial}_t\tens{u} \cdot \tens{v}\,dx\,dt}{\|\tens{v}\|_{L_2(I;\tens{\mathbb{H}}^1(\Omega))}}\\
&=\sup_{0 \neq \tens{v} \in L_2(I;\tens{\mathbb{H}}^1(\Omega))} \frac{\int_I \int_\Omega (\identity \otimes \tens{\Pi}) \tens{\partial}_t\tens{u} \cdot \tens{v}\,dx\,dt}{\|\tens{v}\|_{L_2(I;\tens{\mathbb{H}}^1(\Omega))}}\\
&=\sup_{0 \neq \tens{v} \in L_2(I;\tens{\mathbb{H}}^1(\Omega))} \frac{\int_I \int_\Omega  \tens{\partial}_t\tens{u} \cdot (\identity \otimes \tens{\Pi}) \tens{v}\,dx\,dt}{\|\tens{v}\|_{L_2(I;\tens{\mathbb{H}}^1(\Omega))}}\\
&\leq 
 \|\tens{\Pi}\|_{\cL(\tens{\mathbb{H}}^1(\Omega), \tens{\mathbb{H}}^1(\Omega))} 
 \sup_{0 \neq \tens{v} \in L_2(I;\tens{\mathcal H}^1(\Omega))} \frac{\int_I \int_\Omega  \tens{\partial}_t\tens{u} \cdot  \tens{v}\,dx\,dt}{\|\tens{v}\|_{L_2(I;\tens{\mathbb{H}}^1(\Omega))}}\\
 &= \|\tens{\Pi}\|_{\cL(\tens{\mathbb{H}}^1(\Omega), \tens{\mathbb{H}}^1(\Omega))} \|\tens{\partial}_t\tens{u}\|_{L_2(I;\tens{\mathcal H}^1(\Omega)')}.
\end{align*}
So for $H^1(I;\tens{\mathcal H}^{1}(\Omega)) \supset (\tens{u}_k)_k \rightarrow \tens{u} \in \tens{\mathcal X}:= L_2(I;\tens{\mathcal H}^1(\Omega)) \cap H^1(I;\tens{\mathcal H}^1(\Omega)')$,
$(\tens{u}_k)_k$ is a Cauchy sequence in $\tens{X}:=L_2(I;\tens{\mathcal H}^1(\Omega)) \cap H^1(I;\tens{\mathbb{H}}^1(\Omega)')$, 
and therefore convergent to some $\tens{z} \in \tens{X}$. But then $(\tens{u}_k)_k \rightarrow \tens{z}$ also in $ \tens{\mathcal X}$, or $\tens{z}=\tens{u}$,
and $\|\tens{u}\|_{\tens{X}} = \lim_{n \rightarrow \infty} \|\tens{u}_k\|_{\tens{X}} \lesssim \lim_{n \rightarrow \infty} \|\tens{u}_k\|_{\tens{\mathcal X}}=\|\tens{u}\|_{\tens{\mathcal X}}$, which completes the proof.
\end{proof}

\begin{remark}[no-slip conditions]\label{rem:noslip} 
By replacing  the boundary conditions 
$\tens{u} \cdot \tens{n}= 0$ and $(\identity-\tens{n} \tens{n}^\top)\tens[2]{T}(\nu \tens{u},p) \tens{n}=0$  on  $I\times \partial \Omega$ in \eqref{free-slip-equations}
by $\tens{u}= 0$ on  $I\times \partial \Omega$, one obtains the Stokes problem with \emph{no-slip} boundary conditions.
To analyze the corresponding FOSLS, in Section~\ref{sec:solution-space}
the space $\tens{\mathbb{H}}^1(\Omega)$ should be replaced by $\tens{H}^1_0(\Omega)$ and the condition $(\identity-\tens{n} \tens{n}^\top)\tens[2]{w}|_{I\times \partial \Omega}\tens{n}=0$ should be omitted from the definition of $\mathscr{Z}$,
and in Section~\ref{sec:div-free} the space $\tens{\mathcal H}^1(\Omega)=\{\tens{u} \in \tens{\mathbb{H}}^1(\Omega)\colon \divv \tens{u}=0\}$ 
should be replaced by $\{\tens{u} \in \tens{H}^1_0(\Omega)\colon \divv \tens{u}=0\}$.

The closure of both $\tens{\mathcal H}^1(\Omega)$ and that of $\{\tens{u} \in \tens{H}^1_0(\Omega)\colon \divv \tens{u}=0\}$ in $\tens{L}_2(\Omega)$
is given by $\tens{\mathcal H}^0(\Omega)=\{\tens{u} \in \tens{L}_2(\Omega) \colon \divv \tens{u} =0,\, \tens{u}\cdot\tens{n}=0 \text{ on } \partial\Omega\}$ (see~\cite{75.5}).
Since the $\tens{L}_2(\Omega)$-orthogonal projector $\tens{\Pi}$ onto $\tens{\mathcal H}^0(\Omega)$ does not preserve no-slip boundary conditions, and so $\tens{\Pi} \not\in \cL(\tens{H}^1_0(\Omega),\tens{H}^1_0(\Omega))$, Proposition~\ref{prop1} does not generalize to the no-slip case, and so well-posedness of the corresponding FOSLS in the no-slip case is \emph{not} guaranteed. Numerical results reported in Section~\ref{sec:numerics} seem to suggest that for these boundary conditions this FOSLS is indeed not well-posed.
\end{remark}

\begin{proposition} \label{propje}
Let $\Omega \subset \R^n$ be such that for any $q \in L_{2,0}(\Omega)$, the solution $u \in H^1(\Omega) \cap L_{2,0}(\Omega)$ of
$$
\int_\Omega \nabla u \cdot \nabla v \,d x=\int_\Omega q v \,d x \quad(v \in H^1(\Omega) \cap L_{2,0}(\Omega))
$$
is in $H^2(\Omega)$, with $\|u \|_{H^2(\Omega)} \lesssim \|q\|_{L_2(\Omega)}$.
Then 
$\tens{\Pi} \in \cL(\tens{\mathbb{H}}^1(\Omega),\tens{\mathbb{H}}^1(\Omega))$.
\end{proposition}

\begin{proof} As shown in, e.g., \cite[Ch.~1, Thm. 1.4]{258} or \cite[Thm.~2.7]{75.5}, one has the following {\em Helmholtz decomposition}
\begin{equation} \label{Helm}
\tens{L}_2(\Omega)=\tens{\mathcal H}^0(\Omega) \oplus^{\perp} \nabla (H^1(\Omega) \cap L_{2,0}(\Omega)).
\end{equation}
Given $\tens{u} \in \tens{L}_2(\Omega)$, $\nabla z=(\identity-\tens{\Pi}) \tens{u}$ is the solution of $\int_\Omega (\tens{u}-\nabla z)\cdot \nabla v\,dx=0$ ($v \in H^1(\Omega) \cap L_{2,0}(\Omega)$).
When $\tens{u} \in \tens{\mathbb{H}}^1(\Omega)$, this $z$ solves $\int_\Omega \nabla z \cdot \nabla v \,dx=\int_\Omega \divv \tens{u} \,v\,dx $ ($v \in H^1(\Omega) \cap L_{2,0}(\Omega)$), and so thanks to our assumption on $\Omega$,
$$
\|\nabla z\|_{\tens{H}^1(\Omega)} \eqsim \|z\|_{H^2(\Omega)} \lesssim \|\divv \tens{u}\|_{\tens{L}_2(\Omega)} \lesssim \|\tens{u}\|_{\tens{H}^1(\Omega)},
$$
i.e., $\identity-\tens{\Pi}$, and thus $\tens{\Pi}$, is bounded on $\tens{\mathbb{H}}^1(\Omega)$.
\end{proof}

\begin{remarks}\label{rem:convex} The assumption on $\Omega$ made in Proposition~\ref{propje} is known to be satisfied when $\Omega$ is convex \cite[Thm.~3.2.1.3]{77}, or when it has a $C^2$-boundary \cite[Rem.~3.1.2.4]{77}. (The regularity results in \cite{77} are formulated for the Neumann problem $-\Delta u+\lambda u =q $ on $\Omega$, $\nabla u \cdot \tens{n}=0$ on $\partial\Omega$ for some $\lambda>0$, but because the difference with the Laplacian is an operator of lower order it is well-known that this difference is harmless).

Numerical results presented in Section~\ref{sec:numerics} indicate that our FOSLS is not (fully) well-posed in case of a domain with a re-entrant corner.
\end{remarks}

\subsection{Well-posedness of the FOSLS} \label{Swellposedness}
The following theorem is the main result of this work. 

\begin{theorem} \label{thm:main}
Let 
$\tens{\Pi} \in \cL(\tens{\mathbb{H}}^1(\Omega),\tens{\mathbb{H}}^1(\Omega))$. Then for the operator ${\bf G}$ defined by 
$$
{\bf G}(\tens{u},\tens[2]{w},p)=(\tens[2]{w}+\tens[2]{T}(\nu \tens{u},p),\tens{\partial}_t \tens{u}+\tens{\divv}_x \tens[2]{w},\divv_x \tens{u},\tens{u}(0,\cdot)),
$$
it holds that
$$
\|{\bf G}(\tens{u},\tens[2]{w},p)\|_{\tens[2]{L}_2(I \times \Omega) \times \tens{L}_2(I \times \Omega)\times H^1(I;L_{2,0}(\Omega)) \times \tens{L}_2(\Omega)} \eqsim \|(\tens{u},\tens[2]{w},p)\|_{\mathscr{Z} \times L_2(I;L_{2,0}(\Omega))}
$$
for all $(\tens{u},\tens[2]{w},p) \in \mathscr{Z} \times L_2(I;L_{2,0}(\Omega))$.
\end{theorem}

\begin{proof}
The boundedness of ${\bf G}$ follows from the definition of $\mathscr{Z}$, specifically concerning the term $\|\tens{u}(0,\cdot)\|_{L_2(\Omega)}$ from
\be \label{25}
 \|\tens{u}(0,\cdot)\|_{\tens{L}_2(\Omega)} \lesssim \|\tens{\partial}_t \tens{u}\|_{L_2(I;\tens{\mathbb{H}}^1(\Omega)')}+
 \|\tens{u}\|_{L_2(I;\tens{\mathbb{H}}^1(\Omega))} \lesssim \|(\tens{u},\tens[2]{w})\|_\mathscr{Z}.
\ee
Here the first inequality follows from \cite[Ch.~XVIII, \S1, Thm.~1]{63} using that $\tens{\mathbb{H}}^1(\Omega) \hookrightarrow \tens{L}_2(\Omega)$ with dense embedding, and the second inequality follows from Lemma~\ref{lem2}.
The other direction, i.e.~\eqref{17}, amounts to showing that
 for $(\tens{u},\tens[2]{w},p) \in \mathscr{Z} \times L_2(I;L_{2,0}(\Omega))$,
\be \label{5}
\|\tens{u}\|_{L_2(I;\tens{\mathbb{H}}^1(\Omega))}+\|\tens[2]{w}\|_{\tens[2]{L}_2(I \times \Omega)}+\|p\|_{L_2(I;L_{2,0}(\Omega))} \lesssim \|{\bf G}(\tens{u},\tens[2]{w},p)\|,
\ee
where $\|{\bf G}(\tens{u},\tens[2]{w},p)\|$ will be used as shorthand notation for the norm of ${\bf G}(\tens{u},\tens[2]{w},p)$ from the statement.

Given $(\tens{u},\tens[2]{w},p) \in \mathscr{Z} \times L_2(I;L_{2,0}(\Omega))$ we start with finding a `lift' of $\divv_x \tens{u}$.
There exists an operator $\divv^+ \in \cL(L_{2,0}(\Omega),\tens{H}^1_0(\Omega))$ with $\divv \divv^+=\identity$. An example of such a right-inverse of $\divv$ is given by the mapping $h \mapsto \tens{v}$ where $(\tens{v},q)$ solves the stationary Stokes problem
$$
\left\{
\begin{array}{rcll}
-\tens{\Delta}\tens{v} +\tens{\nabla} q& = & 0 & \text{ on }  \Omega,\\
\divv \tens{v}& = &h & \text{ on } \Omega,\\
\tens{v} & = & \tens{0} & \text{ on } \partial\Omega.
\end{array}
\right.
$$
Setting $\tens{\ell}:=\divv^+_x \divv_x \tens{u}$, we have
\be \label{2}
\begin{split}
\|\tens{\ell}\|_{L_2(I;\tens{\mathbb{H}}^1(\Omega))}+\|\tens{\ell}\|_{H^1(I;\tens{L}_2(\Omega))} & \lesssim
\|\tens{\ell}\|_{H^1(I;\tens{H}^1_0(\Omega))}\\ 
& \lesssim \|\divv_x \tens{u}\|_{H^1(I;L_{2,0}(\Omega))} \leq \|{\bf G}(\tens{u},\tens[2]{w},p)\|,
\end{split}
\ee
and so, again using the first inequality from \eqref{25}, also $\|\tens{\ell}(0,\cdot)\|_{\tens{L}_2(\Omega)} \lesssim  \|{\bf G}(\tens{u},\tens[2]{w},p)\|$.

We set $\tens{z}:=\tens{u}-\tens{\ell}$, which is in $L_2(I;\tens{\mathcal H}^1(\Omega)) \cap H^1(I;\tens{\mathbb{H}}^1(\Omega)')$ by \eqref{2} and Lemma~\ref{lem2}, and, for $\tens{v} \in L_2(I;\tens{\mathbb{H}}^1(\Omega))$, set
$$
\tens{f}(\tens{v}):=\int_I \int_\Omega \tens{\partial}_t \tens{z} \cdot \tens{v}+\tfrac12 \tens[2]{T}(\nu \tens{z},p):\tens[2]{D}(\tens{v})\,dx\,dt.
$$
The boundedness of $\gamma_0$ provided by Theorem~\ref{thm1} shows that $\tens{z}(0,\cdot) \in \tens{\mathcal H}^0(\Omega)$.
Below we will show that
\begin{align} \label{8}
&\|\tens{f}\|_{L_2(I;\tens{\mathbb{H}}^1(\Omega)')} \lesssim \|{\bf G}(\tens{u},\tens[2]{w},p)\|, \\ \label{9}
&\|\tens{z}\|_{L_2(I;\tens{\mathbb{H}}^1(\Omega))}+\|p\|_{L_2(I;L_{2,0}(\Omega))} \lesssim \|\tens{f}\|_{L_2(I;\tens{\mathbb{H}}^1(\Omega)')} + \|\tens{z}(0,\cdot)\|_{\tens{L}_2(\Omega)},
\end{align}
which together with
\begin{align*}
\|\tens{z}(0,\cdot)\|_{\tens{L}_2(\Omega)} &\leq \|\tens{u}(0,\cdot)\|_{\tens{L}_2(\Omega)} + \|\tens{\ell}(0,\cdot)\|_{\tens{L}_2(\Omega)} \lesssim \|{\bf G}(\tens{u},\tens[2]{w},p)\|,\\
\|\tens{u}\|_{L_2(I;\tens{\mathbb{H}}^1(\Omega))}&\leq \|\tens{z}\|_{L_2(I;\tens{\mathbb{H}}^1(\Omega))}+\|\tens{\ell}\|_{L_2(I;\tens{\mathbb{H}}^1(\Omega))}
\lesssim  \|\tens{z}\|_{L_2(I;\tens{\mathbb{H}}^1(\Omega))}+ \|{\bf G}(\tens{u},\tens[2]{w},p)\|,\\
\|\tens[2]{w}\|_{\tens[2]{L}_2(I \times \Omega)} & \leq \|\tens[2]{w}+\tens[2]{T}(\nu \tens{u},p)\|_{\tens[2]{L}_2(I \times \Omega)}+\|\tens[2]{T}(\nu \tens{u},p)\|_{\tens[2]{L}_2(I \times \Omega)}\\
& \lesssim  \|{\bf G}(\tens{u},\tens[2]{w},p)\|+\|\tens{u}\|_{L_2(I;\tens{\mathbb{H}}^1(\Omega))}+\|p\|_{L_2(I;L_{2,0}(\Omega))}
\end{align*}
completes the proof of \eqref{5}, and so of the theorem.

Splitting
$$
\tens{f}(\tens{v})
=\underbrace{\int_I \int_\Omega \tens{\partial}_t \tens{u} \cdot \tens{v}+\tfrac12 \tens[2]{T}(\tens{\nu u},p):\tens[2]{D}(\tens{v})\,dx\,dt
}_{\tens{f}^{(\tens{u},p)}(\tens{v}):=}-
\underbrace{\int_I \int_\Omega \tens{\partial}_t \tens{\ell} \cdot \tens{v}+\tfrac\nu2 \tens[2]{D}(\tens{\ell}):\tens[2]{D}(\tens{v})\,dx\,dt}_{\tens{f}^{\tens{\ell}}(\tens{v}):=},
$$
it holds that $\|\tens{f}^{\tens{\ell}}\|_{L_2(I;\tens{\mathbb{H}}^1(\Omega)')} \lesssim  \|{\bf G}(\tens{u},\tens[2]{w},p)\|$ by \eqref{2}.
Writing 
$$
\tens{f}^{(\tens{u},p)}(\tens{v})=
\int_I \int_\Omega \tens{\partial}_t \tens{u} \cdot \tens{v}-\tfrac12 \tens[2]{w}:\tens[2]{D}(\tens{v})\,dx\,dt+
\int_I \int_\Omega \tfrac12 (\tens[2]{w}+\tens[2]{T}(\nu \tens{u},p)):\tens[2]{D}(\tens{v})\,dx\,dt.
$$
the second term is bounded by some multiple of $\|\tens[2]{w}+\tens[2]{T}(\nu \tens{u},p)\|_{\tens[2]{L}_2(I \times \Omega)} \|\tens{v}\|_{L_2(I;\tens{\mathbb{H}}^1(\Omega))}$.
By symmetry of $\tens[2]{w}$ the first term is equal to
$\int_I \int_\Omega \tens{\partial}_t \tens{u} \cdot \tens{v}- \tens[2]{w}:\tens[2]{\nabla}_x \tens{v}\,dx\,dt$.
For $(\tens{u},\tens[2]{w}) \in \mathcal{Z}_2$, we have
\begin{align*}
\big|\int_I \int_\Omega \tens{\partial}_t \tens{u} \cdot \tens{v}- \tens[2]{w}:\tens[2]{\nabla}_x \tens{v}\,dx\,dt\big|
&\stackrel{\eqref{4}}{=}
\big|\int_I \int_\Omega (\tens{\partial}_t \tens{u} +\tens{\divv}_x \tens[2]{w}) \cdot \tens{v}\,dx\,dt\big|
\\ &\,\, \lesssim \|\tens{\partial}_t \tens{u} +\tens{\divv}_x \tens[2]{w}\|_{\tens{L}_2(I \times \Omega)} \|\tens{v}\|_{L_2(I;\tens{\mathbb{H}}^1(\Omega))}.
\end{align*}
which extends by density to $(\tens{u},\tens[2]{w}) \in \mathscr{Z}_2$, and thus to $(\tens{u},\tens[2]{w}) \in \mathscr{Z}$, which concludes the proof of  \eqref{8}.

For $\tens{v} \in L_2(I;\tens{\mathcal H}^1(\Omega))$, we have
\mbox{$\tfrac12\int_I \int_\Omega p\identity :\tens[2]{D}(\tens{v})\,dx\,dt=\int_I \int_\Omega p \divv_x \tens{v}\,dx\,dt=0$,} and so
$$
\tens{f}(\tens{v})=\int_I \int_\Omega \tens{\partial}_t \tens{z} \cdot \tens{v}+\tfrac{\nu}{2} \tens[2]{D}(\tens{z}):  \tens[2]{D}(\tens{v}) \,dx\,dt.
$$
Korn's second inequality,  which is valid on Lipschitz domains (\cite{239.05}), shows that there exists a constant $\varrho$ such that
\begin{equation}\label{eq:korn}
 \int_\Omega \tens[2]{D}(\tens{v}):  \tens[2]{D}(\tens{v}) \,dx + \varrho \|\tens{v}\|^2_{\tens{L}_2(\Omega)}
\gtrsim \|\tens{v}\|_{\tens{H}^1(\Omega)}^2 \quad(\tens v \in \tens{H}^1(\Omega)).
\end{equation}
Since besides this G{\aa}rding inequality also $\tens z(0,\cdot) \in \tens{\mathcal H}^0(\Omega)$, from an application of Theorem~\ref{thm1} we conclude that
\be \label{10}
\|\tens{z}\|_{L_2(I;\tens{\mathcal H}^1(\Omega))}+\|\tens{\partial}_t \tens{z}\|_{L_2(I;\tens{\mathcal H}^1(\Omega)')}
\lesssim \|f\|_{L_2(I;\tens{\mathcal H}^1(\Omega)')}+\|z(0,\cdot)\|_{L_2(\Omega)},
\ee
where, thanks to the assumption of $\tens{\Pi} \in \cL(\tens{\mathbb{H}}^1(\Omega),\tens{\mathbb{H}}^1(\Omega))$, 
 an application of Proposition~\ref{prop1} shows that
\be \label{11}
 \|\tens{\partial}_t \tens{z}\|_{L_2(I;\tens{\mathbb{H}}^1(\Omega)')} \eqsim \|\tens{\partial}_t \tens{z}\|_{L_2(I;\tens{\mathcal H}^1(\Omega)')}.
\ee

By definition of $\tens{f}$, for $\tens{v} \in L_2(I;\tens{\mathbb{H}}^1(\Omega))$ we have
$$
\int_I \int_\Omega p \divv_x\tens{v} \,dx\,dt=\int_I \int_\Omega \tens{\partial}_t \tens{z}\cdot \tens{v} +\tfrac{\nu}{2} \tens[2]{D}(\tens{z}): \tens[2]{D}(\tens{v})\,dx\,dt-\tens{f}(\tens{v}).
$$
On any Lipschitz domain,
\be \label{eeninfsup}
\inf_{0 \neq q \in L_2(I;L_{2,0}(\Omega))} \sup_{0 \neq \tens{v} \in L_2(I;\tens{\mathbb{H}}^1(\Omega))}
\frac{|\int_I \int_\Omega q \divv_x\tens{v} \,dx\,dt|}{\|q\|_{ L_2(I;L_{2,0}(\Omega))} \|\tens{v}\|_{L_2(I;\tens{\mathbb{H}}^1(\Omega))}} >0,
\ee
being even true with $\tens{\mathbb{H}}^1(\Omega)$ reading as $\tens{H}^1_0(\Omega)$ (e.g. \cite[Ch.~1, Rem.~1.4(ii)]{258}).
Consequently, we have
\begin{align*}
\|p\|_{L_2(I;L_{2,0}(\Omega))} &\lesssim \sup_{0 \neq \tens{v} \in L_2(I;\tens{\mathbb{H}}^1(\Omega))}
\frac{|\int_I \int_\Omega \tens{\partial}_t \tens{z}\cdot \tens{v} +\tfrac12 \tens[2]{D}(\tens{z}): \tens[2]{D}(\tens{v})\,dx\,dt-\tens{f}(\tens{v})|}{\ \|\tens{v}\|_{L_2(I;\tens{\mathbb{H}}^1(\Omega))}}\\
& \lesssim \| \tens{\partial}_t \tens{z}\|_{L_2(I;\tens{\mathbb{H}}^1(\Omega)')}+\|\tens{z}\|_{L_2(I;\tens{\mathbb{H}}^1(\Omega))}+
\|f\|_{L_2(I;\tens{\mathbb{H}}^1(\Omega)')},
\end{align*}
which together with \eqref{10} and \eqref{11} completes the proof of \eqref{9}.
\end{proof}

\begin{remark} \label{rem:dtdiv}
It would be preferable when in the least-squares functional, the component $\divv_x \tens{u}$ is measured in $L_2(I;L_{2,0}(\Omega))$ instead of in $H^1(I;L_{2,0}(\Omega))$, in which case also the condition $\divv_x \tens{u} \in H^1(I;L_{2,0}(\Omega))$ in the definition of solution space could be removed.
Numerical results presented in Section~\ref{sec:numerics} indicate, however, that the corresponding FOSLS is \emph{not} well-posed.

We will apply the FOSLS in the most relevant case that $g=0$, i.e., $\divv_x \tens{u}=0$. To ensure that the measurement in $H^1(I;L_{2,0}(\Omega))$ of the error in $\divv_x \tens{u}$ does \emph{not} increase the smoothness conditions on $\tens{u}$ to achieve a certain convergence rate,  we will apply a trial space that has a commuting diagram property. 
In particular, it comes with a quasi-interpolator into that space that 
maps spatially divergence-free functions to spatially divergence-free functions (cf. (proof of) Proposition~\ref{prop:velo} and Remark~\ref{rem2}).
\end{remark}

\subsection{A slightly modified FOSLS}
A problem with our FOSLS is the pressure space $L_2(I;L_{2,0}(\Omega))$.
A natural approximation space for the pressure is a space of piecewise polynomials w.r.t.~some partition of $I \times \Omega$.
Unless this partition can be decomposed into time-slabs, i.e.~strips of type $[t_i,t_{i+1}] \times \Omega$, it is not clear how to impose the constraint $\int_\Omega p(t,x)\,dx=0$ for all $t \in I$ on such piecewise polynomials. To circumvent this problem, we add this constraint to our least-squares functional.

With
\be \label{defM}
(Mp)(t):=\tfrac{1}{\sqrt{|\Omega|}}\int_\Omega p(t,x)\,dx,
\ee
and $(Ev)(t,x):=\tfrac{1}{\sqrt{|\Omega|}} v(t)$, we have $\|M\|_{\cL(L_2(I \times \Omega),L_2(I))}=1=\|E\|_{\cL(L_2(I),L_2(I \times \Omega))}$ and $\ran (\identity - E M) \subset L_2(I;L_{2,0}(\Omega))$.
After extending the operator ${\bf G}$, which was defined on $\mathscr{Z} \times L_2(I;L_{2,0}(\Omega))$, to an operator on $\mathscr{Z} \times L_2(I \times \Omega)$, we define
$$
{\bf \bar{G}}(\tens{u},\tens[2]{w},p):=({\bf G}(\tens{u},\tens[2]{w},p),Mp).
$$

\begin{corollary} \label{corolbarG}
Let 
$\tens{\Pi} \in \cL(\tens{\mathbb{H}}^1(\Omega),\tens{\mathbb{H}}^1(\Omega))$. Then with
$$
\mathscr{F}:=\tens[2]{L}_2(I \times \Omega) \times \tens{L}_2(I \times \Omega)\times H^1(I;L_{2,0}(\Omega)) \times \tens{L}_2(\Omega) \times L_2(I),
$$
it holds that
$$
\|{\bf \bar{G}}(\tens{u},\tens[2]{w},p)\|_{\mathscr{F}} \eqsim \|(\tens{u},\tens[2]{w},p)\|_{\mathscr{Z} \times L_2(I\times \Omega)}
$$
for all $(\tens{u},\tens[2]{w},p) \in \mathscr{Z} \times L_2(I\times \Omega)$.
\end{corollary}

\begin{proof} 
The boundedness of ${\bf \bar{G}}$ follows from that of ${\bf G}$, which in Theorem~\ref{thm:main} was demonstrated for its domain being $\mathscr{Z} \times L_2(I;L_{2,0}(\Omega))$ but which immediately extends to the larger domain $\mathscr{Z} \times L_2(I \times \Omega)$, together with $\|M\|_{\cL(L_2(I \times \Omega),L_2(I))}=1$.

The other direction follows from
\begin{align*}
\|(\tens{u},\tens[2]{w},p)\|_{\mathscr{Z} \times L_2(I\times \Omega)} &\lesssim \|(\tens{u},\tens[2]{w})\|_{\mathscr{Z}}+\|(\identity - E M)p\|_{L_2(I;L_{2,0}(\Omega))}+\|Mp\|_{L_2(I)}\\
& \stackrel{\text{\makebox[0pt]{Thm.~\ref{thm:main}}}}{\lesssim}\|{\bf G}(\tens{u},\tens[2]{w},(\identity - E M)p)\|+\|Mp\|_{L_2(I)} \\
& \leq \|{\bf G}(\tens{u},\tens[2]{w},p)\|+\|{\bf G}(0,0,E M p)\|+\|Mp\|_{L_2(I)}\\
&\lesssim \|{\bf G}(\tens{u},\tens[2]{w},p)\|+\|E M p\|_{L_2(I \times \Omega)}+\|Mp\|_{L_2(I)}\\
&\lesssim \|{\bf G}(\tens{u},\tens[2]{w},p)\|+\|Mp\|_{L_2(I)} \lesssim \|{\bf \bar{G}}(\tens{u},\tens[2]{w},p)\|_{\mathscr{F}},
\end{align*}
where $\|\cdot\|$ abbreviates again the norm of Theorem~\ref{thm:main}.
\end{proof}

Above corollary can be formulated by saying that ${\bf \bar{G}}$ is a boundedly invertible linear mapping from $\mathscr{Z} \times L_2(I\times \Omega)$ onto its range in $\mathscr{F}$. Restricting to the \emph{most relevant case} of $\divv_x \tens{u}=0$ on $I \times \Omega$, i.e.,
 $$
 g=0 \text{ in } \eqref{free-slip-equations},
 $$
in order to know for which right-hand sides there exists an exact solution of the first order system reformulation of \eqref{free-slip-equations},
$$
{\bf \bar{G}}(\tens{u},\tens[2]{w},p)
=
(\tens[2]{w}+\tens[2]{T}(\nu\tens{u},p), \tens{\partial}_t \tens{u} +\tens{\divv}_x \tens[2]{w},\divv_x \tens{u},\tens{u}(0,\cdot),Mp)=(0,\tens{f},0,\tens{u}_0,0)
$$
we study $\ran {\bf \bar{G}}|_{\{(\tens{u},\tens[2]{w},p) \in \mathscr{Z} \times L_2(I\times\Omega)  \colon \tens[2]{w}+\tens[2]{T}(\nu\tens{u},p)=0,\,\divv_x \tens{u}=0,\,Mp=0\}}$.

\begin{proposition} \label{proprange} Let 
$\tens{\Pi} \in \cL(\tens{\mathbb{H}}^1(\Omega),\tens{\mathbb{H}}^1(\Omega))$. Then
\begin{align*}
\ran {\bf \bar{G}}&|_{\{(\tens{u},\tens[2]{w},p) \in \mathscr{Z} \times L_2(I\times\Omega)  \colon  \tens[2]{w}+\tens[2]{T}(\nu\tens{u},p)=0,\,\divv_x \tens{u}=0,\,Mp=0\}}\\&\hspace*{5em}=
\{0\} \times \tens{L}_2(I \times \Omega) \times \{0\} \times \tens{\mathcal H}^0(\Omega) \times \{0\}.
\end{align*}
\end{proposition} 

\begin{proof} For $(\tens{u},\tens[2]{w})\in \mathscr{Z}$, we have $ \tens{\partial}_t \tens{u} +\tens{\divv}_x \tens[2]{w} \in L_2(I \times \Omega)$ and
$\tens{u} \in L_2(I;\tens{\mathbb{H}}^1(\Omega)) \cap H^1(I;\tens{\mathbb{H}}^1(\Omega)')$ by Lemma~\ref{lem2}. When $\divv_x \tens{u}=0$, then even
$\tens{u} \in L_2(I;\tens{\mathcal H}^1(\Omega)) \cap H^1(I;\tens{\mathcal H}^1(\Omega)')$, and so $\tens{u}(0,\cdot) \in \tens{\mathcal H}^0(\Omega)$ by Theorem~\ref{thm1}, which completes the proof of `$\subset$'.

To show `$\supset$', given $\tens{f} \in \tens{L}_2(I \times \Omega)$ and $\tens{u}_0 \in \tens{\mathcal H}^0(\Omega)$, let
$\tens{u} \in L_2(I;\tens{\mathcal H}^1(\Omega)) \cap H^1(I;\tens{\mathcal H}^1(\Omega)')$  with $\tens{u}(0,\cdot)=\tens{u}_0$ be the solution of
$$
\int_I \int_\Omega \tens{\partial}_t \tens{u} \cdot \tens{v}+\tfrac{\nu}{2} \tens[2]{D}(\tens{u}):  \tens[2]{D}(\tens{v}) \,dx\,dt=
\int_I \int_\Omega \tens{f} \cdot \tens{v}\,dx\,dt \quad(\tens{v} \in L_2(I;\tens{\mathcal H}^1(\Omega))),
$$
whose existence follows from Theorem~\ref{thm1} (which is applicable due to~\eqref{eq:korn}).
Recall from Proposition~\ref{prop1} that $\tens{u} \in H^1(I;\tens{\mathbb H}^1(\Omega)')$.
Let $p \in L_2(I;L_{2,0}(\Omega))$, i.e., $p \in L_2(I\times \Omega)$ with $Mp=0$, be the solution of
\be \label{press}
\int_I \int_\Omega p \divv_x\tens{v} \,dx\,dt=\int_I \int_\Omega \tens{\partial}_t \tens{u}\cdot \tens{v} +\tfrac{\nu}{2} \tens[2]{D}(\tens{u}): \tens[2]{D}(\tens{v})\,dx\,dt-\int_I \int_\Omega \tens{f} \cdot \tens{v}\,dx\,dt
\ee
($\tens v \in L_2(I;\tens{\mathbb{H}}^1(\Omega))$), whose existence follows from \eqref{eeninfsup} and the fact the right-hand side vanishes for $\tens{v} \in L_2(I;\tens{\mathcal H}^1(\Omega))$ being the kernel of $\divv_x$ in $L_2(I;\tens{\mathbb H}^1(\Omega))$. 
Setting $\tens[2]{w}:=-\tens[2]{T}(\nu \tens{u},p)=-\nu\tens[2]{D}(\tens{u})+p \identity$, the symmetry of $\tens[2]{w}$ shows that \eqref{press} is equivalent to
\be \label{26}
\int_I \int_\Omega \tens{\partial}_t \tens{u}\cdot \tens{v} -\tens[2]{w} : \tens[2]{\nabla}_x \tens{v} \,dx\,dt=\int_I \int_\Omega \tens{f} \cdot \tens{v}\,dx\,dt
\quad (v \in L_2(I;\tens{\mathbb{H}}^1(\Omega))).
\ee
Realizing that for smooth $\tens{v}$ that vanish on $\partial I \times \Omega$, 
$\int_I \!\int_\Omega \tens{\partial}_t \tens{u}\cdot \tens{v}\! +\! \tens{u}\cdot \tens{\partial}_t\tens{v}   \,dxdt=0$, 
and recalling that
$(\tens{\partial}_t \tens{u}+\tens{\divv}_x \tens[2]{w})_i$ is the divergence of the vector field $(\tens{u}_i,\tens[2]{w}_{i\cdot})^\top$, by considering \eqref{26} for $\tens{v} \in \tens{\mathcal D}(I \times \Omega)$ we infer that $\tens{\partial}_t \tens{u}+\tens{\divv}_x \tens[2]{w}=\tens{f} \in \tens{L}_2(I \times \Omega)$.
For all $v \in \tens{V}$, \eqref{eq:partial integration} and Remark~\ref{rem:density} give that
\begin{align*}
&\int_I \int_{\partial\Omega} 
(\identity-\tens{n} \tens{n}^\top)\tens[2]{w}\tens{n} \cdot \tens{v} \,ds \,dt=
\int_I \int_{\partial\Omega} 
\tens[2]{w}\tens{n} \cdot \tens{v} \,ds \,dt\\
&=
\int_I \int_\Omega \tens[2]{w}:\tens[2]{\nabla}_x \tens{v}-\tens{\partial}_t \tens{u}\cdot \tens{v}
+(\tens{\partial}_t \tens{u}+ \tens{\divv}_x \tens[2]{w})\cdot \tens{v}\,dx\,dt=0.
\end{align*}
Therefore, we conclude that $(\tens{u},\tens[2]{w},p) \in \mathscr{Z} \times L_2(I;L_{2,0}(\Omega))$ with ${\bf \bar{G}}(\tens{u},\tens[2]{w},p)=(0,\tens{f},0,\tens{u}_0,0)$, which completes the proof.
\end{proof}

\begin{remark} In view of Proposition~\ref{proprange}, one might incorporate the condition $\divv_x \tens{u}=0$ in the specification of the domain of ${\bf \bar{G}}$, i.e., replace $\mathscr{Z}$ by
\begin{align*}
\tilde{\mathscr{Z}}:=\{(\tens{u},\tens[2]{w}) \in L_2(I;\tens{\mathcal H}^1(\Omega))\times L_2(I;L_2(\Omega;\tens[2]{\mathbb{S}})) \colon &
\tens{\partial}_t \tens{u}+\tens{\divv}_x \tens[2]{w} \in \tens{L}_2(I \times \Omega),\\
& (\identity-\tens{n} \tens{n}^\top)\tens[2]{w}|_{I\times \partial \Omega}\tens{n}=0\},
\end{align*}
and omit the third component $\divv_x \tens{u}$ from the definition of ${\bf \bar{G}}$. Clearly, the reason not to do so is that generally it is hard to construct suitable finite dimensional subspaces of $\tilde{\mathscr{Z}}$.
\end{remark}

\section{Numerical approximation} \label{sec:a priori}
The following proposition recalls that least-squares approximations of a well-posed least-squares system are quasi-optimal.
\begin{proposition} \label{prop:approx} Let 
$\tens{\Pi} \in \cL(\tens{\mathbb{H}}^1(\Omega),\tens{\mathbb{H}}^1(\Omega))$. Let $\mathscr{Z}_\delta \times P_\delta$ be a closed subspace of $\mathscr{Z} \times L_2(I \times \Omega)$.
Given $F \in \mathscr{F}$, let $(\tens{u},\tens[2]{w},p)$ or $(\tens{u}_\delta,\tens[2]{w}_\delta,p_\delta)$ be the minimizers over
$\mathscr{Z} \times L_2(I \times \Omega)$ or 
$\mathscr{Z}_\delta \times P_\delta$, respectively, of
$$
\tfrac12 \|F - {\bf \bar{G}}(\cdot,\cdot,\cdot)\|^2_{\mathscr{F}}.
$$
Then with
\begin{align}\label{constantM}
\frak{M}:=\sup_{0 \neq (\tens{\hat{u}},\tens[2]{\hat{w}},\hat{p}) \in {\mathscr{Z} \times L_2(I \times \Omega)}} \|{\bf \bar{G}}(\tens{\hat{u}},\tens[2]{\hat{w}},\hat{p})\|_{\mathscr{F}} / \|(\tens{\hat{u}},\tens[2]{\hat{w}},\hat{p})\|_{\mathscr{Z} \times L_2(I \times \Omega)},\\ \label{constantm}
\frak{m}:=\inf_{0 \neq (\tens{\hat{u}},\tens[2]{\hat{w}},\hat{p}) \in {\mathscr{Z} \times L_2(I \times \Omega)}} \|{\bf \bar{G}}(\tens{\hat{u}},\tens[2]{\hat{w}},\hat{p})\|_{\mathscr{F}} / \|(\tens{\hat{u}},\tens[2]{\hat{w}},\hat{p})\|_{\mathscr{Z} \times L_2(I \times \Omega)},
\end{align}
which satisfy $\frak{M} <\infty$ and $\frak{m}>0$ by Corollary~\ref{corolbarG}, it holds that
\be \label{6}
\begin{split}
\|(\tens{u},\tens[2]{w},p)-&(\tens{u}_\delta,\tens[2]{w}_\delta,p_\delta)\|_{\mathscr{Z} \times L_2(I \times\Omega)}\\ & \leq \tfrac{\frak{M}}{\frak{m}} \inf_{(\tens{\hat{u}},\tens[2]{\hat{w}},\hat{p}) \in \mathscr{Z}_\delta \times P_\delta}
\|(\tens{u},\tens[2]{w},p)-(\tens{\hat{u}},\tens[2]{\hat{w}},\hat{p})\|_{\mathscr{Z} \times L_2(I \times \Omega)}.
\end{split}
\ee
If $F \in \ran {\bf \bar{G}}$, and so in particular if $F \in \{0\} \times \tens{L}_2(I \times \Omega) \times \{0\} \times \tens{\mathcal H}^0(\Omega) \times \{0\}$, then for any $(\tens{\hat{u}},\tens[2]{\hat{w}},\hat{p}) \in \mathscr{Z} \times L_2(I \times\Omega)$, we have the \emph{a posteriori bounds}
$$
\frak{M}^{-1} \|F-{\bf \bar{G}}(\tens{\hat{u}},\tens[2]{\hat{w}},\hat{p})\|_{\mathscr{F}}\leq
\|(\tens{u},\tens[2]{w},p)-(\tens{\hat{u}},\tens[2]{\hat{w}},\hat{p})\|_{\mathscr{Z} \times L_2(I \times \Omega)} \leq 
\frak{m}^{-1} \|F-{\bf \bar{G}}(\tens{\hat{u}},\tens[2]{\hat{w}},\hat{p})\|_{\mathscr{F}}.
$$
\end{proposition}

\begin{proof}
The first statement follows from $(\tens{u},\tens[2]{w},p)$ and $(\tens{u}_\delta,\tens[2]{w}_\delta,p_\delta)$ being the solutions of 
\begin{align*}
\langle {\bf \bar{G}} (\tens{u},\tens[2]{w},p),{\bf \bar{G}} (\tens{\hat{u}},\tens[2]{\hat{w}},\hat{p})\rangle_{\mathscr{F}} &=\langle F,{\bf \bar{G}} (\tens{\hat{u}},\tens[2]{\hat{w}},\hat{p})\rangle_{\mathscr{F}} \quad((\tens{\hat{u}},\tens[2]{\hat{w}},\hat{p}) \in \mathscr{Z} \times L_2(I \times \Omega)),\\
\langle {\bf \bar{G}} (\tens{u}_\delta,\tens[2]{w}_\delta,p_\delta),{\bf \bar{G}} (\tens{\hat{u}},\tens[2]{\hat{w}},\hat{p})\rangle_{\mathscr{F}} &=\langle F,{\bf \bar{G}} (\tens{\hat{u}},\tens[2]{\hat{w}},\hat{p})\rangle_{\mathscr{F}} \quad((\tens{\hat{u}},\tens[2]{\hat{w}},\hat{p}) \in \mathscr{Z}_\delta \times P_\delta)
\end{align*}
respectively, the definitions of $\frak{M}$ and $\frak{m}$, and, e.g., \cite[Remarks~(2.8.5)]{35.7}.

The second statement follows from ${\bf \bar{G}} (\tens{u},\tens[2]{w},p)=F$.
\end{proof}

To apply Proposition~\ref{prop:approx} we need to select suitable finite dimensional subspaces of $\mathscr{Z} \times L_2(I\times\Omega)$ of finite element type. 
We consider spaces of type $\tens{U}_\delta \times \tens[2]{W}_\delta \times P_\delta$, where to achieve conformity
\begin{align} \label{21}
\tens{U}_\delta &\subset \tens{U}:=\{\tens{u} \in L_2(I;\tens{\mathbb{H}}^1(\Omega)) \cap H^1(I;\tens{L}_2(\Omega))\colon \divv_x \tens{u} \in H^1(I;L_{2,0}(\Omega))\},\\ 
\label{22}
\tens[2]{W}_\delta & \subset L_2(I;H_0(\tens{\divv};\Omega,\tens[2]{\mathbb{S}})),
\\
P_\delta & \subset L_2(I \times \Omega),
\end{align}
where
$$
H_0(\tens{\divv};\Omega,\tens[2]{\mathbb{S}}):=\{\tens[2]{v} \in L_2(\Omega,\tens[2]{\mathbb{S}})\colon \tens{\divv}\, \tens[2]{v} \in \tens{L}_2(\Omega),\,(\identity-\tens{n}\tens{n}^\top) \tens[2]{v}|_{\partial\Omega} \tens{n}=0\}.
$$
Recall that the latter condition $(\identity-\tens{n}\tens{n}^\top) \tens[2]{v}|_{\partial\Omega} \tens{n}=0$ is equivalent to $\tens[2] v \tens n \cdot \tens \tau = 0$ for $\tens \tau\perp \tens n$.
By equipping $\tens{U}$ and $L_2(I;H_0(\tens{\divv};\Omega,\tens[2]{\mathbb{S}}))$ with their natural norms, notice that
$$
\tens{U} \times L_2(I;H_0(\tens{\divv};\Omega,\tens[2]{\mathbb{S}})) \hookrightarrow \mathscr{Z}.
$$

When considering $\tens{U}_\delta$ to be a finite element space w.r.t.~a general partition of $I \times \Omega$ into, say, $(n+1)$-simplices, because of the constraint on $\divv_x \tens{u}$ one needs $C^1$-elements which are not very practical.
Therefore, we will consider $\tens{U}_\delta$ (and $\tens[2]{W}_\delta$ and $P_\delta$) to be finite element spaces w.r.t.~(a common) partition of  $I \times \Omega$ into \emph{prismatic elements}, i.e., intervals-in-time times $n$-simplices-in-space.
Such spaces have the advantage that the interelement smoothness conditions needed to be conforming w.r.t.~the non-isotropic spaces at the right-hand sides of \eqref{21} and \eqref{22} can be reduced to their minimum. 
We expect that such a choice can have a beneficial effect on the convergence rates that can be achieved for solutions that exhibit singularities using appropriately \emph{locally refined partitions}.

Yet, in the current work we will restrict ourselves to quasi-uniform \emph{conforming partitions}, and postpone the treatment of locally refined, and so generally necessarily \emph{nonconforming} partitions to forthcoming work.
Although not essential, furthermore for our convenience we consider the situation of a \emph{two-dimensional} domain $\Omega$, and finite element spaces of \emph{lowest order}.

\subsection{Trial space for the velocities}
Let $I_\delta$ be a partition of $I$ into subintervals, and $\Omega_\delta$  be a conforming partition of $\Omega$ into uniformly shape regular triangles.
Let $N_\delta$ or $E_\delta$ denote the union of the vertices or edges of the $K \in \Omega_\delta$.

Let $U_\delta^t$ and $U_\delta^x$ be the spaces of continuous piecewise linears w.r.t.\ $I_\delta$ and $\Omega_\delta$, respectively, and $\tens{U}_\delta^x:=U_\delta^x \times U_\delta^x$.
For any $e \in E_\delta$, let $\tens{\phi}^e =\tens{\phi}^e_\delta \in \tens{C}(\overline{\Omega})$ be the `edge bubble' introduced in \cite{38.77}, see also \cite[\S3.1-3]{35.927}. 
We have $\int_e \tens{\phi}^e \cdot \tens{n} \,ds \neq 0$, and with $\Omega_\delta^e:=\{K \in \Omega_\delta\colon e \subset \partial K\}$, $\supp \tens{\phi}^e=\cup \Omega_\delta^e$. For $K \in \Omega_\delta^e$, $\tens{\phi}^e|_K$ is piecewise linear w.r.t.~the Powell--Sabin split of $K$, and $\divv \tens{\phi}^e|_K \in P_0(K)$.
We define
$$
\tens{U}_\delta:=\underbrace{U^t_\delta \otimes \{\tens{v} \in \tens{U}_\delta^x\colon \tens{v}\cdot \tens{n}=0 \text{ on } \partial\Omega\}}_{\tens{\widehat{U}}_\delta:=} +
U^t_\delta \otimes\Span\{\tens{\phi}^e\colon e \in E_\delta\text{ with }e\not\subset \partial\Omega\}.
$$%
For $\tens{v} \in \tens{U}_\delta^x$, the condition $\tens{v}\cdot \tens{n}=0$ on  $\partial\Omega$ means that for $z \in N_\delta \cap \partial\Omega$, 
$\tens{v}(z)=0$ when $z$ is a corner of $\Omega$, and $(\tens{v} \cdot \tens{n})(z)=0$ otherwise.
We will establish the following \emph{a priori bound} on the approximation error.

\begin{proposition} \label{prop:velo} For $\tens{u} \in \tens{H}^2(I \times \Omega) \cap L_2(I;\tens{\mathbb{H}}^1(\Omega))$ with
$\divv_x \tens{u} \in H^2(I;L_2(\Omega)) \cap H^1(I;H^1(\Omega))$, with $h_\delta:=\max(\max_{J \in I_\delta} \diam J,\max_{K \in \Omega_\delta} \diam K)$ it holds that
$$
\inf_{\tens{v} \in \tens{U}_\delta}\|\tens{u}-\tens{v}\|_{\tens{U}} \lesssim h_\delta \big(\|\tens{u}\|_{\tens{H}^2(I \times \Omega)}+\|\divv_x \tens{u}\|_{H^2(I;L_2(\Omega))}+\|\divv_x \tens{u}\|_{H^1(I;H^1(\Omega))}\big).
$$
In particular, if $\divv_x \tens{u}=0$, then $\inf_{\tens{v} \in \tens{U}_\delta}\|\tens{u}-\tens{v}\|_{\tens{U}} \lesssim h_\delta \|\tens{u}\|_{\tens{H}^2(I \times \Omega)}$.
\end{proposition}

\begin{proof} 
On $\tens{H}^1(\Omega)$ we define a Scott--Zhang type projector $\tens{\hat{\mathcal J}}_\delta$ onto $\tens{U}_\delta^x$ 
that will map $\tens{\mathbb{H}}^1(\Omega)$ onto  $\{\tens{v} \in \tens{U}_\delta^x\colon \tens{v}\cdot \tens{n}=0 \text{ on } \partial\Omega\}$.
Let $\tens{v} \in \tens{H}^1(\Omega)$. For $z \in N_\delta \setminus \partial\Omega$, we define $(\tens{\hat{\mathcal J}}_\delta \tens{v})(z)$ by applying coordinate-wise the standard Scott--Zhang interpolator at an interior node.

For $z \in N_\delta \cap \partial\Omega$, let $e^a, e^b \in E_\delta \cap \partial\Omega$ with $e^a\cap e^b=z$, having normals $\tens{n}^a$ and $\tens{n}^b$, respectively.
For $c \in \{a,b\}$, and with $\tilde{z}^c  \in N_\delta \cap \partial\Omega$ denoting the other endpoint of $e^c$, 
let $\psi_{e^c} \in P_1(e^c)$ be such that $\int_{e^c} \phi_z \psi_{e^c}\,ds=1$ and $\int_{e^c} \phi_{\tilde{z}^c} \psi_{e^c}\,ds=0$, where $\phi_z$ and $\phi_{\tilde{z}^c}$ are the standard nodal basis functions of $U^x_\delta$ associated to $z$ and $\tilde{z}^c$, respectively.
If $z$ is not a corner of $\Omega$,  select $c \in \{a,b\}$ arbitrary, and define $(\tens{\hat{\mathcal J}}_\delta \tens{v})(z)$ by 
$$
(\tens{\hat{\mathcal J}}_\delta \tens{v})(z)\cdot \tens{n}^c=\int_{e^c} \tens{v}\cdot \tens{n}^c\psi_{e^c} \,ds,\quad
(\tens{\hat{\mathcal J}}_\delta \tens{v})(z)\cdot \tens{\tau}^c=\int_{e^c} \tens{v}\cdot \tens{\tau}^c\psi_{e^c} \,ds
$$
with $\tens{\tau}^c$ a unit vector that is orthogonal to $\tens{n}^c$.
For $z$ being a corner of $\Omega$, define $(\tens{\hat{\mathcal J}}_\delta \tens{v})(z)$ by 
$$
(\tens{\hat{\mathcal J}}_\delta \tens{v})(z)\cdot \tens{n}^a=\int_{e^a} \tens{v}\cdot \tens{n}^a\psi_{e^a} \,ds,\quad
(\tens{\hat{\mathcal J}}_\delta \tens{v})(z)\cdot \tens{n}^b=\int_{e^b} \tens{v}\cdot \tens{n}^b\psi_{e^b} \,ds.
$$

Since $\tens{\hat{\mathcal J}}_\delta$ locally reproduces $\tens{P}_1$, standard arguments show that for $K \in \Omega_\delta$ and $m,k \in \{0,1\}$,
$$
\|\tens{v}-\tens{\hat{\mathcal J}}_\delta \tens{v}\|_{\tens{H}^m(K)} \lesssim \diam(K)^{k+1-m} |\tens{v}|_{H^{k+1}(\Omega_\delta(K))},
$$
where $\Omega_\delta(K):=\{K' \in \Omega_\delta\colon K \cap K' \neq \emptyset\}$. Moreover $\tens{\hat{\mathcal J}}_\delta$ preserves vanishing normals, and so maps $\tens{\mathbb{H}}^1(\Omega)$ onto $\{\tens{v} \in \tens{U}_\delta^x\colon \tens{v}\cdot \tens{n}=0 \text{ on } \partial\Omega\}$.

We define  the linear map $\tens{\mathcal J}_\delta \colon \tens{H}^1(\Omega) \rightarrow \tens{U}^x_\delta +\Span\{\tens{\phi}^e \colon e \in E_\delta\}$ by 
\be \label{16}
\tens{\mathcal J}_\delta \tens{v}:=\tens{\hat{\mathcal J}}_\delta \tens{v} + \sum_{e \in E_\delta} \frac{\int_e (\tens{v}-\tens{\hat{\mathcal J}}_\delta \tens{v})\cdot \tens{n}\,ds}{\int_e \tens{\phi}^e \cdot \tens{n}\,ds} \tens{\phi}^e.
\ee
Because $\tens{\mathcal J}_\delta$  locally reproduces $\tens{P}_1$, we have that for  $K \in \Omega_\delta$ and $m,k \in \{0,1\}$,
\be \label{14}
\|\tens{v}-\tens{\mathcal J}_\delta \tens{v}\|_{\tens{H}^m(K)} \lesssim \diam(K)^{k+1-m} |\tens{v}|_{H^{k+1}(\Omega_\delta(K))}.
\ee

For any $K \in \Omega_\delta$ and $\tens{v} \in \tens{H}^1(\Omega)$, it holds that
$$
\int_T \divv \tens{\mathcal J}_\delta \tens{v}\,dx=
\int_{\partial K} \tens{\mathcal J}_\delta \tens{v} \cdot \tens{n}\,dx=
\int_{\partial K} \tens{v} \cdot \tens{n}\,dx=\int_T \divv \tens{v}\,dx,
$$
and so thanks to $\divv \tens{\mathcal J}_\delta \tens{v}$ being piecewise constant w.r.t.~$\Omega_\delta$, with $\mathcal{Q}_\delta$ being the $L_2(\Omega)$-orthogonal projector onto the piecewise constants, we have the \emph{commuting diagram}
\be \label{7}
 \divv \tens{\mathcal J}_\delta=\mathcal{Q}_\delta \divv \quad\text{on } \tens{H}^1(\Omega).
\ee
Finally,
$$
\ran \tens{\mathcal J}_\delta|_{\tens{\mathbb{H}}^1(\Omega)} \subset 
\{\tens{v} \in \tens{U}_\delta^x\colon \tens{v}\cdot \tens{n}=0 \text{ on } \partial\Omega\}+\Span\{\tens{\phi}^e\colon e \in E_\delta \text{ with }e\not\subset\partial\Omega\}).
$$

Let ${\mathcal I}_\delta\colon H^1(I) \mapsto U^t_\delta$ denote the Scott-Zhang projector. Noticing that ${\mathcal I}_\delta \otimes \tens{\mathcal J}_\delta|_{\tens{\mathbb{H}}^1(\Omega)}$ maps into $\tens{U}_\delta$, and $\identity-{\mathcal I}_\delta\otimes \tens{\mathcal J}_\delta=
(\identity-{\mathcal I}_\delta) \otimes \identity+{\mathcal I}_\delta \otimes (\identity-\tens{\mathcal J}_\delta)$, and taking into account \eqref{7}, we have that for $\tens{u} \in \tens{H}^2(I \times \Omega) \cap L_2(I;\tens{\mathbb{H}}^1(\Omega))$ with
$\divv_x \tens{u} \in H^2(I;L_2(\Omega)) \cap H^1(I;H^1(\Omega))$,
\begin{align*}
\inf_{\tens{v} \in \tens{U}_\delta}& \big\{\|\tens{u}-\tens{v}\|_{\tens{H}^1(I\times\Omega)}+\|\divv_x (\tens{u}-\tens{v})\|_{H^1(I;L_{2,0}(\Omega))}\big\} \\
&
\leq \|(\identity-{\mathcal I}_\delta) \otimes \identity \,\tens{u}\|_{\tens{H}^1(I\times\Omega)}+
\|{\mathcal I}_\delta \otimes (\identity-\tens{\mathcal J}_\delta) \tens{u}\|_{\tens{H}^1(I\times\Omega)} +\\
&\quad\,\,\|(\identity-{\mathcal I}_\delta) \otimes \identity \divv_x \tens{u}\|_{H^1(I;L_{2,0}\Omega))}+
\|{\mathcal I}_\delta \otimes (\identity-\mathcal{Q}_\delta)  \divv_x \tens{u}\|_{H^1(I;L_{2,0}\Omega))} 
 \\
&\lesssim h_\delta \big(\|\tens{u}\|_{\tens{H}^2(I \times \Omega)}+\|\divv_x \tens{u}\|_{H^2(I;L_2(\Omega))}+\|\divv_x \tens{u}\|_{H^1(I;H^1(\Omega))}\big). 
\end{align*}
This concludes the proof.
\end{proof}

\begin{remark} \label{rem2}
\emph{Without} the inclusion of the bubbles in the space $\tens{U}_\delta$, i.e., when taking the space $\tens{\widehat{U}}_\delta$, the estimate from Proposition~\ref{prop:velo} reads as
$$
\inf_{\tens{v} \in \tens{\hat{U}_\delta}}\|\tens{u}-\tens{v}\|_{\tens{U}} \lesssim h_\delta\big(
\|\tens{u}\|_{\tens{H}^2(I \times \Omega)}+\|\divv_x \tens{u}\|_{H^2(I;L_2(\Omega))}+ \|\tens{u}\|_{H^2(I;\tens{H}^1(\Omega))} \big)
$$
assuming $\tens{u} \in \tens{H}^2(I \times \Omega) \cap  H^2(I;\tens{H}^1(\Omega)) \cap  L_2(I;\tens{\mathbb{H}}^1(\Omega))$ with 
 $\divv_x \tens{u} \in H^2(I;L_2(\Omega))$.
This follows from estimating
\begin{align*}
&\|\divv_x (\identity-{\mathcal I}_\delta \otimes \tens{\hat{\mathcal J}}_\delta) \tens{u} \|_{H^1(I;L_{2,0}(\Omega))} \\
&=
\|\divv_x \big((\identity-{\mathcal I}_\delta) \otimes \identity+{\mathcal I}_\delta \otimes (\identity-\tens{\hat{\mathcal J}}_\delta)\big) \tens{u} \|_{H^1(I;L_{2,0}(\Omega))}\\ 
& \lesssim
\| (\identity -{\mathcal I}_\delta) \otimes \identity \divv_x \tens{u} \|_{H^1(I;L_{2,0}(\Omega))}+
\| {\mathcal I}_\delta \otimes (\identity-  \tens{\hat{\mathcal J}}_\delta) \tens{u} \|_{H^1(I;\tens{H}^1(\Omega))}
\\
&\lesssim
h_\delta \big( \|\divv_x \tens{u}\|_{H^2(I;L_{2,0}(\Omega))}+\|\tens{u} \|_{H^1(I;\tens{H}^2(\Omega))}\big).
\end{align*}
\end{remark}

\subsection{Trial space for the stress tensor}
We define $\tens[2]{W}_\delta=W^t_\delta \otimes \tens[2]{W}^x_\delta$, where $W^t_\delta$ is the space of piecewise constants w.r.t.~$I_\delta$, and 
$\tens[2]{W}^x_\delta$ is a finite element subspace of $H_0(\tens{\divv};\Omega,\tens[2]{\mathbb{S}})$ w.r.t.~$\Omega_\delta$.
The construction of finite element spaces for symmetric stress tensors is an intensively studied topic, see e.g.~\cite{168.75,14.4,14.5,139.6,139.7,79}.
We will use the recently introduced finite element from \cite[Remark 5.4]{38.78} because it has only 9 degrees of freedom, which is (much) less than in other constructions. Moreover, as we will see, this element has a commuting diagram property which will allow us to bound $\tens{\divv}_x$ of the approximation error in $\tens[2]{w}$ in terms of $\tens{\divv}_x \tens[2]{w}$ only (cf.~\eqref{24}). This will be beneficial in Theorem~\ref{corolletje}.

Given a triangle $K$, let $E(K)$ denote the set of its edges and  $R(K)$ the Clough--Tocher decomposition of $K$ into $3$ subtriangles by connecting the vertices with the centroid.
Let
\begin{align*}
\tens{V}_K&:=\{\tens{v} \in \tens P_0(R(K))\colon \tens{v}\cdot \tens{n} \text{ is continuous across the internal edges}\},\\
\tens[2]{W}_K&:=\{\tens[2]{w} \in P_1(R(K),\tens[2]{\mathbb{S}})\colon \tens[2]{w}\tens{n}  \text{ is continuous across the internal edges},\\
&\qquad \tens{\divv} \tens[2]{w} \in \tens{V}_K,\,\tens[2]{w}\tens{n}\cdot\tens{\tau} \in P_0(e)\,(e \in E(K))\},
\end{align*}
where $\tens \tau\perp \tens n$.

As stated in \cite{38.78}, the space $\tens[2]{W}_K$ has dimension $9$, and the zeroth order moment of $\tens[2]{w}\tens{n}\cdot\tens{\tau}$, and zeroth and first order moments of $\tens[2]{w}\tens{n}\cdot\tens{n}$ over each of the edges $e \in E(K)$ can be used as the $3*(1+2)$ degrees of freedom.

Knowing that a $\tens[2]{w}$ with $\tens[2]{w}|_K \in \tens[2]{W}_K$ ($K \in \Omega_\delta$)
is in $H(\tens{\divv};\Omega,\tens[2]{\mathbb{S}})$ whenever $\tens[2]{w}\tens{n}$ is continuous across any edge between triangles in $\Omega_\delta$, we conclude that a valid set of degrees of freedom for
$$
\tens[2]{W}^x_\delta:=\{\tens[2]{w} \in H_0(\tens{\divv};\Omega,\tens[2]{\mathbb{S}})\colon \tens[2]{w}|_{K} \in \tens[2]{W}_K\,(K \in \Omega_\delta)\}
$$
is given by the union of the aforementioned moments over all $e \in E_\delta$, setting the zeroth order moments of 
$\tens[2]{w}\tens{n}\cdot\tens{\tau}$ over the edges $e \subset \partial\Omega$ to zero.

\begin{proposition} \label{prop:stress}
For $\tens[2]{w} \in L_2\big(I;H_0(\tens{\divv};\Omega,\tens[2]{\mathbb{S}}) \cap  H^1(\Omega,\tens[2]{\mathbb{S}})\big) \cap H^1(I;H(\tens{\divv};\Omega,\tens[2]{\mathbb{S}}))$
 with $\tens{\divv}_x \tens[2]{w} \in L_2(I;\tens{H}^1(\Omega)) $, it holds that
\begin{align*}
\inf_{\tens[2]{v} \in \tens[2]{W}_\delta} \|\tens[2]{w}-\tens[2]{v}&\|_{L_2(I;H(\tens{\divv};\Omega,\tens[2]{\mathbb{S}}))} \\
&\lesssim h_\delta \big(
\|\tens[2]{w}\|_{H^1(I;H(\tens{\divv};\Omega,\tens[2]{\mathbb{S}}))}+
\|\tens[2]{w}\|_{L_2(I;H^1(\Omega,\tens[2]{\mathbb{S}}))}+
\|\tens{\divv}_x \tens[2]{w}\|_{L_2(I;\tens{H}^1(\Omega))}\big).
\end{align*}
\end{proposition}

\begin{proof}
For $\tens[2]{w} \in H^1(K,\tens[2]{\mathbb{S}})$, we define the projection $\tens[2]{\mathcal{P}}_K \tens[2]{w}$ onto $\tens[2]{W}_K$ by
$$
\int_e (\tens[2]{w}-\tens[2]{\mathcal{P}}_K \tens[2]{w}) \tens{n}\cdot \tens{\tau} \,P_0(e)\,ds=0=
\int_e (\tens[2]{w}-\tens[2]{\mathcal{P}}_K \tens[2]{w}) \tens{n}\cdot \tens{n} \,P_1(e)\,ds \quad(e \in E(K)).
$$
Using that $\tens[2]{{\mathcal{P}}}_K$ reproduces $\tens[2]{W}_K$ and so in particular all constant symmetric tensors, one infers that
\be \label{23}
\|(\identity-\tens[2]{{\mathcal{P}}}_K)\tens[2]{w}\|_{L_2(K,\tens[2]{\mathbb{S}})} \lesssim \diam(K) \|\tens[2]{w}\|_{H^1(K,\tens[2]{\mathbb{S}})}.
\ee

On $H^1(\Omega,\tens[2]{\mathbb{S}})$ the corresponding global projector $\tens[2]{{\mathcal{P}}}_\delta$  is defined by
$(\tens[2]{{\mathcal{P}}}_\delta \tens[2]{w})|_K =\tens[2]{{\mathcal{P}}}_K \tens[2]{w}|_K$ ($K \in \Omega_\delta$).
It maps $H^1(\Omega,\tens[2]{\mathbb{S}}) \cap H_0(\tens{\divv};\Omega,\tens[2]{\mathbb{S}})$ onto $\tens[2]{W}^x_\delta$.

We define $\tens{RM}_K:=\tens{P}_0(K)+(x_2,-x_1)^\top P_0(K)$.
Both $\tens{V}_K$ and $\tens{RM}_K$ are $3$-dimensional, and, as stated in \cite{38.78}, the integrals over $K$ against a basis for $\tens{RM}_K$ is a valid set of degrees of freedom for 
$\tens{V}_K$. 
In particular, the \emph{biorthogonal} projector $\tens{{\mathcal{Q}}}_K$ onto $\tens{RM}_K$ with $\ran (\identity - \tens{{\mathcal{Q}}}_K) \perp_{\tens{L}_2(K)} \tens{V}_K$ is well-defined. Its adjoint $\tens {\mathcal{Q}}^*_K$ is the projector onto $ \tens{V}_K$  with
$\ran (\identity - \tens {\mathcal{Q}}^*_K) \perp_{\tens{L}_2(K)} \tens{RM}_K$.

Using that for $\tens{v} \in \tens{RM}_K$,   $\tens{v}\cdot\tens{\tau} \in P_0(e)$ for $e \in E(K)$ and $\tens[2] D(\tens{v})=0$, for $\tens[2]{w} \in H^1(K,\tens[2]{\mathbb{S}})$
it holds that
$$
\int_T \tens{\divv}(\tens[2]{w} - \tens[2]{{\mathcal{P}}}_K \tens[2]{w}) \cdot \tens{v} \,dx=
\int_{\partial K} (\tens[2]{w} - \tens[2]{{\mathcal{P}}}_K \tens[2]{w})\tens{n} \cdot \tens{v}\,ds-\tfrac12 \int_T (\tens[2]{w} - \tens[2]{{\mathcal{P}}}_K \tens[2]{w}): \tens[2]D(\tens{v})\,dx=0,
$$
i.e.,  
\begin{align*}
0 = \langle \tens{\divv}(\identity-\tens[2]{{\mathcal{P}}}_K) \tens[2]{w},\tens{RM}_K\rangle_{\tens{L}_2(K)}&=
\langle \tens{{\mathcal{Q}}}_K^*\tens{\divv}(\identity-\tens[2]{{\mathcal{P}}}_K) \tens[2]{w}, \tens{RM}_K\rangle_{\tens{L}_2(K)}\\&=
\langle \tens{{\mathcal{Q}}}_K^*\tens{\divv}\tens[2]{w}-\tens{\divv}\tens[2]{{\mathcal{P}}}_K \tens[2]{w}, \tens{RM}_K\rangle_{\tens{L}_2(K)}, 
\end{align*}
where we used that $\ran \tens{\divv}|_{\tens[2]{W}_K} \subset \tens V_T$. Thanks to $\tens{{\mathcal{Q}}}_K^*\tens{\divv}\tens[2]{w}-\tens{\divv}\tens[2]{{\mathcal{P}}}_K \tens[2]{w} \in \tens V_T$ we conclude that
$$
\tens{\divv}\,\tens[2]{{\mathcal{P}}}_K=\tens{{\mathcal{Q}}}_K^* \,\tens{\divv} \quad \text{on } H^1(K,\tens[2]{\mathbb{S}}).
$$
From this \emph{commuting diagram property}, using that $ \tens{{\mathcal{Q}}}_K^*$ reproduces all constant vectors we infer that
\be \label{24}
\|\tens{\divv}(\tens[2]{w} - \tens[2]{{\mathcal{P}}}_K \tens[2]{w})\|_{\tens{L}_2(K)}=
\|(\identity - \tens{{\mathcal{Q}}}_K^*) \tens{\divv} \,\tens[2]{w}\|_{\tens{L}_2(K)} \lesssim \diam(K) \|\tens{\divv} \,\tens[2]{w}\|_{\tens{H}^1(K)}.
\ee

 By writing
$\identity - {\mathcal{P}}_\delta \otimes \tens[2]{{\mathcal{P}}}_\delta=(\identity-{\mathcal{P}}_\delta) \otimes\identity+{\mathcal{P}}_\delta \otimes (\identity-\tens[2]{{\mathcal{P}}}_\delta)$, where 
${\mathcal{P}}_\delta$ denotes the $L_2(I)$-orthogonal projector onto $W^t_\delta$, the proof follows from
\eqref{23} and \eqref{24}.
\end{proof}

\subsection{Trial space for the pressure, and combined a priori error estimate}

For the pressure, we employ the standard trial space of piecewise constants.

\begin{proposition} \label{prop:pressure} For $P_\delta$ being the space of piecewise constants w.r.t.~$I_\delta \times \Omega_\delta$, and $p \in H^1(I \times \Omega)$, it holds that
$$
\inf_{q \in P_\delta}  \|p-q\|_{L_2(I \times \Omega)} \lesssim h_\delta \|p\|_{H^1(I \times \Omega)}.
$$
\end{proposition}

By combining Propositions~\ref{proprange}, \ref{prop:approx}, \ref{prop:velo}, \ref{prop:stress}, and \ref{prop:pressure}, we obtain the following a priori error estimate.

\begin{theorem} \label{corolletje}
 Let 
$\tens{\Pi} \in \cL(\tens{\mathbb{H}}^1(\Omega),\tens{\mathbb{H}}^1(\Omega))$. 
For $(\tens{f},\tens{u}_0) \in \tens{L}_2(I \times \Omega) \times  \tens{\mathcal H}^0$, let $(\tens{u},\tens[2]{w},p) \in \mathscr{Z}\times L_2(I\times \Omega)$ be the solution of  ${\bf \bar{G}}(\tens{u},\tens[2]{w},p)=(0,\tens{f},0,\tens{u}_0,0)$.
Then if $(\tens{u},p,\tens{f}) \in \tens{H}^2(I \times \Omega) \times H^1(I \times \Omega) \times \tens{H}^1(I \times \Omega)$, 
 the minimizer
$(\tens{u}_\delta,\tens[2]{w}_\delta,p_\delta)$ over $\tens{U}_\delta \times \tens[2]{W}_\delta \times P_\delta$ of
$$
\tfrac12 \|(0,\tens{f},0,\tens{u}_0,0) - {\bf \bar{G}}(\tens{u}_\delta,\tens[2]{w}_\delta,p_\delta)\|^2_{\tens[2]{L}_2(I \times \Omega) \times \tens{L}_2(I \times \Omega)\times H^1(I;L_{2,0}(\Omega)) \times \tens{L}_2(\Omega) \times L_2(I)}
$$
satisfies
$$
\|(\tens{u},\tens[2]{w},p)-(\tens{u}_\delta,\tens[2]{w}_\delta,p_\delta)\|_{\mathscr{Z} \times L_2(I \times \Omega)} \lesssim h_\delta \big(\|\tens{u}\|_{\tens{H}^2(I \times \Omega) }+\|p\|_{H^1(I \times \Omega)}+\|\tens{f}\|_{\tens{H}^1(I \times \Omega)}\big).
$$
\end{theorem}

\begin{proof} Using that $\tens{\divv}_x \tens[2]{w}=\tens{f}-\tens{\partial}_t \tens{u}$, the results follows.
\end{proof}

\section{Numerical experiments}\label{sec:numerics}
In this section we consider $I=(0,1)$ and $\Omega$ to be either the unit square $(0,1)^2$ or the L-shaped domain $(-1,1)^2\setminus[-1,0]^2$. 

To construct test problems for \eqref{free-slip-equations} with $\nu=1$, $g=0$, and a known solution, we need to find a $\tens{u}$ with $\divv_x \tens{u}=0$ on $I \times \Omega$, and $\tens{u} \cdot \tens{n}=0=\tens[2]{D}(\tens{u}) \tens{n}\cdot \tens{\tau}$ on $I \times \partial \Omega$. By taking $\tens{u}=\tens{\curl}_x \psi$, the first condition is satisfied, whereas the second and third conditions read as $\partial_{\tens{\tau}} \psi=0$ on $I \times \partial \Omega$, and $4  \tens{n}_1 \tens{n}_2 \partial_{x_1}  \partial_{x_2} \psi+
(\tens{n}_2^2-\tens{n}_1^2) (\partial_{x_2}^2-\partial_{x_1}^2) \psi =0$ on $I \times \partial \Omega$, respectively.

On both domains we prescribe the exact solution
$$
\tens{u}(t,x_1,x_2) := \exp(-t) \, \tens{\curl}_x \frac{\sin(\pi x_1)\sin(\pi x_2)}{\pi}
$$
which satisfies aforementioned conditions,
$$
p(t,x_1,x_2) := \exp(-t) \sin\big(\pi(x_1-x_2)\big),
$$
and $\tens[2]{w}:=-\tens[2]{T}(\nu\tens{u},p)$, and take the data $\tens{u}_0(x_1,x_2):=u(0,x_1,x_2)$, $\tens{f}:=\tens{\partial}_t \tens{u} -\nu \tens{\Delta}_{x} \tens{u} +\nabla_{\bf x}\, p$, and $F:=(0,\tens{f},0,\tens{u}_0,0)$ correspondingly.
 
We compute the least-squares approximation $(\tens{u}_\delta,\tens[2]{w}_\delta,p_\delta)$ of $(\tens{u},\tens[2]{w},p)$ in the triple of lowest order finite element spaces $\tens{U}_\delta\times \tens[2]{W}_\delta\times P_\delta$
specified in Section~\ref{sec:a priori} w.r.t.~a partition of $I \times \Omega$ from a sequence of partitions created by subsequent uniform refinements starting from 
$I$ times the initial triangulations of $\Omega$ from Figure~\ref{eq:geometries}.
\begin{figure}[h]
\begin{tikzpicture}
\tikzmath{\c=4;}
\draw[line width = 0.25mm] (0,0) -- (\c,0) -- (\c,\c) -- (0,\c) -- cycle;
\draw[line width = 0.25mm] (0,0) -- (\c,\c);
\draw[line width = 0.25mm] (\c,0) -- (0,\c);
\fill[red] (\c/2,\c/2) circle (3pt);
\end{tikzpicture}
\qquad\qquad
\begin{tikzpicture}
\tikzmath{\c=2;}
\draw[line width = 0.25mm] (0,0) -- (\c,0) -- (\c,\c) -- (0,\c) -- cycle;
\draw[line width = 0.25mm] (0,0) -- (\c,\c);
\draw[line width = 0.25mm] (\c,0) -- (0,\c);
\draw[line width = 0.25mm] (0,-\c) -- (\c,-\c) -- (\c,0) -- (0,0) -- cycle;
\draw[line width = 0.25mm] (0,-\c) -- (\c,0);
\draw[line width = 0.25mm] (\c,-\c) -- (0,0);
\draw[line width = 0.25mm] (-\c,0) -- (0,0) -- (0,\c) -- (-\c,\c) -- cycle;
\draw[line width = 0.25mm] (-\c,0) -- (0,\c);
\draw[line width = 0.25mm] (0,0) -- (-\c,\c);
\fill[red] (\c/2,\c/2) circle (3pt);
\fill[red] (-\c/2,\c/2) circle (3pt);
\fill[red] (\c/2,-\c/2) circle (3pt);
\end{tikzpicture}
\caption{Unit square $\Omega=(0,1)^2$ and L-shaped domain $\Omega=(-1,1)^2\setminus[-1,0]^2$ with initial triangulation. The newest vertices are marked with a dot.}\label{eq:geometries}
\end{figure}
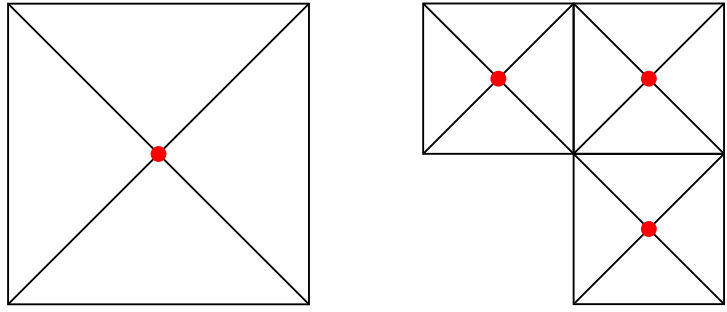
Each uniform refinement consists of splitting each prismatic element into 8 subelements by applying one bisection in time and two recursive newest vertex bisections in space.

The approximations are computed as the solution in $\tens{U}_\delta\times\tens[2]{W}_\delta\times P_\delta$ of 
$$
\langle {\bf \bar{G}} (\tens{u}_\delta,\tens[2]{w}_\delta,p_\delta),{\bf \bar{G}} (\tens{\hat{u}},\tens[2]{\hat{w}},\hat{p})\rangle_{\mathscr{F}} =\langle F,{\bf \bar{G}} (\tens{\hat{u}},\tens[2]{\hat{w}},\hat{p})\rangle_{\mathscr{F}} \quad((\tens{\hat{u}},\tens[2]{\hat{w}},\hat{p}) \in \tens{U}_\delta\times\tens[2]{W}_\delta\times P_\delta).
$$
While the corresponding Galerkin matrix is not sparse due to the presence of the operator $M$ from \eqref{defM}, the matrix can easily be applied in linear complexity.
We thus compute the discretizations via PCG with diagonal preconditioner. 

Note that on the L-shaped domain well-posedness of our FOSLS is not guaranteed by our theory as the equivalence of norms from Proposition~\ref{prop1} has only been shown for convex domains or domains with a $C^2$-boundary (see Remark~\ref{rem:convex}).
With $\frak{M}_\delta$ or $\frak{m}_\delta$ defined as the supremum or infimum over $\tens{U}_\delta\times \tens[2]{W}_\delta\times P_\delta$ of the right-hand sides of \eqref{constantM} or \eqref{constantm}, respectively, in view of the estimate from \eqref{6}
in Table~\ref{tab:ratios} we provide the values of the quotients $\frak{M}_\delta/\frak{m}_\delta$ for decreasing mesh-sizes $h_\delta=\max(\max_{J \in I_\delta} \diam J,\max_{K \in \Omega_\delta} \diam K)$. 
These can be computed by solving a discrete general eigenvalue problem. On the unit square, the constants remain uniformly bounded in correspondence with Corollary~\ref{corolbarG}. While the constants keep slightly increasing on the L-shaped domain as the mesh is refined, suggesting that  generally Corollary~\ref{corolbarG} is \emph{not} valid without $H^2(\Omega)$-regularity of the Poisson problem (cf. Proposition~\ref{propje}),
they remain relatively small and one might hope for reasonable approximations $(\tens{u}_\delta,\tens[2]{w}_\delta,p_\delta)$. 

As announced in Remark~\ref{rem:dtdiv}, in Table~\ref{tab:ratios} we further consider the case where the component $\divv_x \tens{u}$ in the least-squares functional is measured in $L_2(I;L_{2,0}(\Omega))$ instead of in $H^1(I;L_{2,0}(\Omega))$. 
In this case, we also remove the condition $\divv_x \tens{u}$ $\in$ $H^1(I;L_{2,0}(\Omega))$ in the definition of the solution space $\mathscr{Z}$.
The constants $\frak{M}_\delta$ and $\frak{m}_\delta$ are defined accordingly.
The increasing values of $\frak{M}_\delta/\frak{m}_\delta$ indicate that this modified FOSLS is not well-posed. 
Finally, as discussed in Remark~\ref{rem:noslip}, we briefly consider no-slip boundary conditions adapting the boundary conditions of the solution space and its finite-dimensional trial spaces accordingly.
The values of the quotient $\frak{M}_\delta/\frak{m}_\delta$ give an indication that for these boundary conditions the FOSLS is not well-posed.

\begin{center}
\begin{table}[h]
\caption{Ratios $\frak{M}_\delta/\frak{m}_\delta$ for decreasing mesh-sizes $h_\delta$ on the unit square $\Omega=(0,1)^2$ and the L-shape $\Omega=(-1,1)^2\setminus[-1,0]^2$.
On the square, we measure the component $\divv_x \tens u$ in $H^1(I;L_{2,0}(\Omega))$ or $L_2(I;L_{2,0}(\Omega))$, and consider slip or no-slip boundary conditions.}
\label{tab:ratios}
\begin{tabular}{l | c | c | c | c | c | c}
Mesh-size & $2^0$ & $2^{-1}$ & $2^{-2}$ & $2^{-3}$ & $2^{-4}$ & $2^{-5}$\\
\hline
Square (slip, $\partial_t\divv_x \tens{u}\in L_2$) & 3.73 & 6.75 & 6.81 & 6.82 & 6.82 & 6.82\\
L-shape (slip, $\partial_t\divv_x \tens{u}\in L_2$) & 7.65 & 9.23 & 10.73 & 12.22 & 13.59 & 14.81\\ 
Square (slip, $\divv_x \tens{u}\in L_2$) & 3.73 & 6.88 & 7.37 & 8.21 & 10.96 & 18.88\\
Square (no-slip, $\partial_t\divv_x \tens{u}\in L_2$) & 5.92 & 7.94  & 10.62 & 13.36 & 15.28 & 16.72\\
\end{tabular}
\end{table}
\end{center}

\subsection{Experiment on unit square}\label{sec:square}
In Figure~\ref{fig:square} we plot the least-squares estimator
$\eta_\delta:=\|F-{\bf \bar{G}}(\tens{u}_\delta,\tens[2]{w}_\delta,p_\delta)\|_{\mathscr{F}}$.
As predicted in Section~\ref{sec:a priori}, the convergence rate on the unit square is given by $\mathcal{O}({\rm dofs}^{-1/3}) = \mathcal{O}(h_\delta)$. 
We also plot the errors 
$\tens{u}-\tens{u}_\delta$, $\tens[2]{w}-\tens[2]{w}_\delta$, 
$\tens{\partial}_t(\tens{u}-\tens{u}_\delta)+\tens{\divv}_x(\tens[2]{w}-\tens[2]{w}_\delta)=
\tens{f}-(\tens{\partial}_t \tens{u}_\delta+\tens{\divv}_x\tens[2]{w}_\delta)$,
and $p-p_\delta$ measured in the norms $\big(\|\cdot\|^2_{L_2(I;\tens{H}^1(\Omega))}+\|\divv_x \cdot\|_{H^1(I;L_{2}(\Omega))}^2\big)^{1/2}$,  $\|\cdot\|_{\tens[2]{L}_2(I \times \Omega)}$, $\|\cdot\|_{\tens{L}_2(I \times \Omega)}$ and $\|\cdot\|_{L_2(I\times\Omega)}$. 
In Corollary~\ref{corolbarG} the sum of  these quantities were proven to be equivalent to the estimator. 
We thus expect and indeed observe the same convergence rate $\mathcal{O}({\rm dofs}^{-1/3})$.

\begin{remark} In this, as well as in other examples that we have tested, we got similar, actually slightly better convergence behavior (in terms of a multiplicative constant) when we omitted the bubbles from the space $\tens{U}_\delta$, i.e., when we replaced $\tens{U}_\delta$ by $\tens{\widehat{U}}_\delta$.
A possible explanation might be that for the regular triangulations that we have used, there exists a suitable quasi-interpolator 
to the space of continuous piecewise linear vector fields that maps divergence-free functions to divergence-free functions, similar to $\tens{\mathcal J}_\delta$ from \eqref{16} for this space enriched with edge bubbles.
\end{remark}

\subsection{Experiment on L-shape}\label{sec:lshape}
In Figure~\ref{fig:lshape} we plot again the least-squares estimator as well as the errors measured in different norms. 
Although not covered by the theory anymore, we observe in each case the same convergence rate $\mathcal{O}({\rm dofs}^{-1/3})$.


\begin{center}
\begin{figure}
\begin{tikzpicture}
\begin{loglogaxis}[
width = 0.75\textwidth,
clip = true,
xlabel= {degrees of freedom},
ylabel = {estimator and errors},
grid = major,
legend pos = south west,
legend entries = {$\eta_\delta$\\$\tens{u}-\tens{u}_\delta$\\$\tens[2]{w}-\tens[2]{w}_\delta$\\$\tens f-\tens\divv(\tens u_\delta,\tens[2] w_\delta)$\\$p-p_\delta$\\} 
]
\addplot[line width = 0.25mm,color=red,mark=*] table[x=dofs,y=eta,/pgf/number format/read comma as period] {geo0dir1vel1str1.csv};
\addplot[line width = 0.25mm,color=blue,mark=x] table[x=dofs,y=err_u] {geo0dir1vel1str1.csv};
\addplot[line width = 0.25mm,color=green,mark=+] table[x=dofs,y=err_w] {geo0dir1vel1str1.csv};
\addplot[line width = 0.25mm,color=cyan,mark=asterisk] table[x=dofs,y=err_pde] {geo0dir1vel1str1.csv};
\addplot[line width = 0.25mm,color=brown,mark=triangle*] table[x=dofs,y=err_p] {geo0dir1vel1str1.csv};

\draw[line width = 0.25mm] (axis cs:1e5,1e0/1.5) -- (axis cs:1e6,1e0/1.5) -- (axis cs:1e6,0.1^0.333/1.5) -- cycle;
\node[above] at (axis cs:10^5.5,1/1.5) {$1$};
\node[right] at (axis cs:1e6,0.1^0.1666/1.5) {$\frac13$};
\end{loglogaxis}
\end{tikzpicture}
\caption{\label{fig:square}
Least-squares estimator
$\eta_\delta=\|F-{\bf \bar{G}}(\tens{u}_\delta,\tens[2]{w}_\delta,p_\delta)\|_{\mathscr{F}}$ and
errors 
$\tens{u}-\tens{u}_\delta$, $\tens[2]{w}-\tens[2]{w}_\delta$, 
$\tens{f}-(\tens{\partial}_t \tens{u}_\delta+\tens{\divv}_x\tens[2]{w}_\delta)$,
and $p-p_\delta$ measured in $\big(\|\cdot\|^2_{L_2(I;\tens{H}^1(\Omega))}+\|\divv_x \cdot\|_{H^1(I;L_{2}(\Omega))}^2\big)^{1/2}$,  $\|\cdot\|_{\tens[2]{L}_2(I \times \Omega)}$, $\|\cdot\|_{\tens{L}_2(I \times \Omega)}$ and $\|\cdot\|_{L_2(I\times\Omega)}$, respectively, for the problem on the unit square from Section~\ref{sec:numerics}.}
\end{figure}
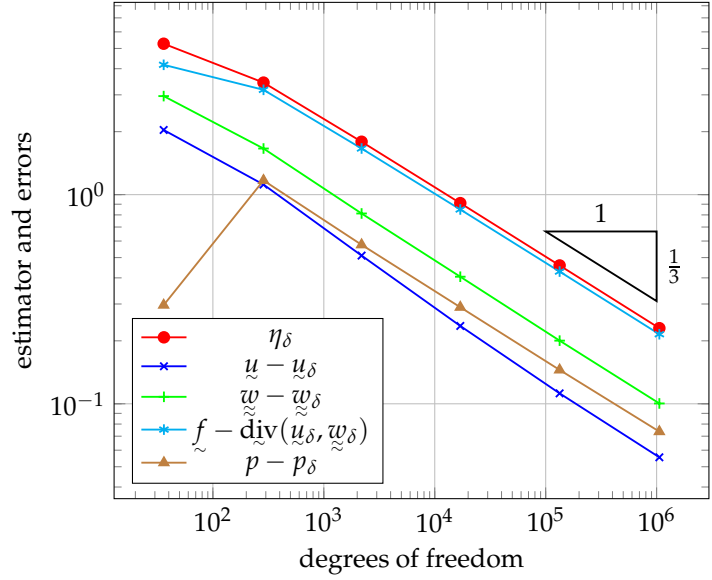
\end{center}


\begin{center}
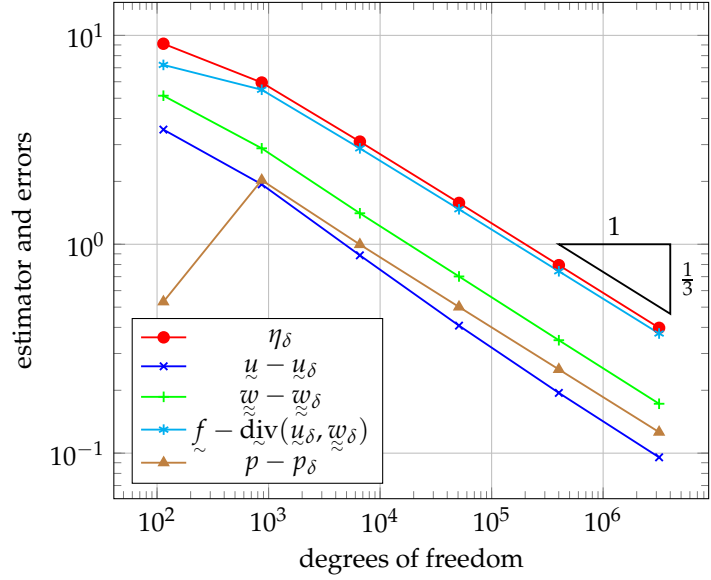
\begin{figure}
\begin{tikzpicture}
\begin{loglogaxis}[
width = 0.75\textwidth,
clip = true,
xlabel= {degrees of freedom},
ylabel = {estimator and errors},
grid = major,
legend pos = south west,
legend entries = {$\eta_\delta$\\$\tens{u}-\tens{u}_\delta$\\$\tens[2]{w}-\tens[2]{w}_\delta$\\$\tens f-\tens\divv(\tens u_\delta,\tens[2] w_\delta)$\\$p-p_\delta$\\} 
]
\addplot[line width = 0.25mm,color=red,mark=*] table[x=dofs,y=eta,/pgf/number format/read comma as period] {geo1dir1vel1str1c1.csv};
\addplot[line width = 0.25mm,color=blue,mark=x] table[x=dofs,y=err_u] {geo1dir1vel1str1c1.csv};
\addplot[line width = 0.25mm,color=green,mark=+] table[x=dofs,y=err_w] {geo1dir1vel1str1c1.csv};
\addplot[line width = 0.25mm,color=cyan,mark=asterisk] table[x=dofs,y=err_pde] {geo1dir1vel1str1.csv};
\addplot[line width = 0.25mm,color=brown,mark=triangle*] table[x=dofs,y=err_p] {geo1dir1vel1str1c1.csv};
\draw[line width = 0.25mm] (axis cs:1e5*4,1e0) -- (axis cs:1e6*4,1e0) -- (axis cs:1e6*4,0.1^0.333) -- cycle;
\node[above] at (axis cs:10^5.5*4,1) {$1$};
\node[right] at (axis cs:1e6*4,0.1^0.1666) {$\frac13$};
\end{loglogaxis}
\end{tikzpicture}
\caption{\label{fig:lshape}
Least-squares estimator
$\eta_\delta:=\|F-{\bf \bar{G}}(\tens{u}_\delta,\tens[2]{w}_\delta,p_\delta)\|_{\mathscr{F}}$ and
 errors 
$\tens{u}-\tens{u}_\delta$, $\tens[2]{w}-\tens[2]{w}_\delta$, 
$\tens{f}-(\tens{\partial}_t \tens{u}_\delta+\tens{\divv}_x\tens[2]{w}_\delta)$,
and $p-p_\delta$ measured in $\big(\|\cdot\|^2_{L_2(I;\tens{H}^1(\Omega))}+\|\divv_x \cdot\|_{H^1(I;L_{2}(\Omega))}^2\big)^{1/2}$,  $\|\cdot\|_{\tens[2]{L}_2(I \times \Omega)}$, $\|\cdot\|_{\tens{L}_2(I \times \Omega)}$ and $\|\cdot\|_{L_2(I\times\Omega)}$, respectively, for the problem on the L-shape from Section~\ref{sec:numerics}. }
\end{figure}
\end{center}

\bibliographystyle{modalpha}
\bibliography{../../ref}

\end{document}